\documentclass[12pt]{amsart}

\usepackage[margin=3cm]{geometry}
\usepackage[english]{babel}
\usepackage{hyperref}
\hypersetup{hidelinks}
\usepackage{amsmath}
\usepackage{amsthm}
\usepackage{amssymb}
\usepackage{bbm}
\usepackage{ucs}
\usepackage{enumerate}
\usepackage{tikz}
\usepackage{tikz-cd}
\usepackage[utf8]{inputenc}
\usepackage[T1]{fontenc}
\usepackage{mathtools}
\usepackage{dsfont}
\usepackage{faktor}
\usepackage{float}
\usepackage{needspace}
\AtBeginDocument{\def\MR#1{}}

\theoremstyle{plain}
\newtheorem{theorem}{Theorem}[section]

\newtheorem{thmx}{Theorem}

\newtheorem{lemma}[theorem]{Lemma}
\newtheorem{proposition}[theorem]{Proposition}
\newtheorem{corollary}[theorem]{Corollary}

\theoremstyle{definition}
\newtheorem{definition}[theorem]{Definition}
\newtheorem{example}[theorem]{Example}

\newtheorem{remark}[theorem]{Remark}

\newcommand{\C}{\mathbb{C}}
\newcommand{\N}{\mathbb{N}}

\newcommand{\Z}{\mathbb{Z}}

\newcommand{\ot}{\otimes}
\newcommand{\id}{{\rm id}}

\newcommand{\F}{\mathbb{F}}
\newcommand{\Aut}{{\rm Aut}}
\newcommand{\Ad}{{\rm Ad}}

\DeclareMathOperator{\Rad}{Rad}
\DeclareMathOperator{\Sub}{Sub}

\DeclareMathOperator{\Prob}{Prob}

\DeclareMathOperator{\IRA}{IRA}
\DeclareMathOperator{\qIRA}{qIRA}
\DeclareMathOperator{\IRS}{IRS}

\DeclareMathOperator{\SRA}{\text{-}SRA}

\newcommand{\supp}{\mathrm{supp}}

\title{Random $C^*$-subalgebras}
\date{}
\author{Pierre Fima}
\address{Pierre Fima
\newline
Universit\'e Paris Cit\'e and Sorbonne Universit\'e, CNRS, IMJ-PRG, F-75013 Paris, France.}
\email{pierre.fima@imj-prg.fr}
\thanks{P.F. is partially supported by the CNRS IEA GAOA and ANR-25-CE40-5010 (CroCQG).
L.T. is supported by China Scholarship Council.
This research was
funded in whole or in part by the Austrian Science Fund (FWF)
10.55776/ESP4189024.}

  \author{Hanna Oppelmayer}
\address{Hanna Oppelmayer
\newline
University of Innsbruck}
\email{hanna.oppelmayer@uibk.ac.at} 
  
 \author{Leyan Tao}
\address{Leyan Tao
\newline
Universit\'e Paris Cit\'e and Sorbonne Universit\'e, CNRS, IMJ-PRG, F-75013 Paris, France.}
\email{tao@imj-prg.fr}

\begin{document}

\begin{abstract}
Given a discrete countable group $\Gamma$, we introduce the concept of random subalgebras of the reduced group $C^*$-algebra $C_r^*(\Gamma)$. This extends the existing theory of invariant random subgroups and the one of invariant random subalgebras of the group von Neumann algebra ${\rm L}(\Gamma)$.  As applications, we provide a characterisation of $C^*$-simplicity, a condition for amenability of the group and show the existence and uniqueness of a maximal nuclear $\Gamma$-invariant unital $C^*$-subalgebra of $C^*_r(\Gamma)$. Moreover, we give a topological characterisation of exactness of any separable unital $C^*$-algebra. Finally, we study exotic random subalgebras of $C^*_r(\Gamma)$ and super-exotic random subalgebras of ${\rm L}(\Gamma)$, that is, random subalgebras which do not arise purely from random subgroups of $\Gamma$ and not from random subalgebras of $C^*_r(\Gamma)$, respectively.
\end{abstract}

\maketitle

\section{Introduction}
\noindent The aim of this work is to develop a theory of \textit{random $C^*$-subalgebras}, which extends the existing notions for groups and for von Neumann algebras. Let us first give a brief introduction to the concept for groups.

\vspace{0.2cm}

\noindent Let $\Gamma$ be a countable group. Let $\Sub(\Gamma)$ denote the space of subgroups of $\Gamma$, endowed with the Chabauty topology. An \textit{invariant random subgroup} (IRS) is a Borel probability measure on $\Sub(\Gamma)$ which is invariant under the conjugation action $\Gamma \curvearrowright \Sub(\Gamma)$.

\vspace{0.2cm}

\noindent This concept of invariant random subgroups was introduced by Vershik \cite{V1, V2}. 
In 2014, Abért–Glasner–Virág \cite{AGV14} developed the theory further and coined the modern name. In particular, they showed that every IRS arises as the push-forward measure of some pmp action on a standard probability space under the stabiliser map. Before their work, stabilisers of random points appeared in Bergeron--Gaboriau \cite{BG04}, while  Aldous--Lyons \cite{AL07} considered random networks, which in the case of Schreier graphs naturally relate to random subgroups.
\vspace{0.2cm}

\noindent IRSs have turned out to be very useful and have been studied in plenty of works. Let us provide here only a short excerpt, not aiming for completeness. For instance, Bowen \cite{Bowen14} used IRS to prove full entropy spectra of stationary actions of the free group. Similar for lamplighter groups by Hartman--Yadin \cite{HartmanYadin}. 
Also, the famous  Stuck--Zimmer Theorem  \cite{StuckZimmer1994} can be reformulated in terms of IRSs. Further rigidity results arising from IRSs can be found in  \cite{HartmanTamuz16},    \cite{Gel}, \cite{s7}, \cite{Glasner17}.
In contrast to Stuck--Zimmer rigidity,  Bowen \cite{BFree} showed that every nonabelian free group bears plenty of different IRSs.  

\vspace{0.2cm}

\noindent Instead of considering the space of subgroups, one may consider the space of von Neumann subalgebras of the group von Neumann algebra ${\rm L}(\Gamma)$, denoted here as well by $\Sub({\rm L}(\Gamma))$, viewed with the action $\Gamma \curvearrowright \Sub({\rm L}(\Gamma))$ by conjugation via the left-regular representation.
The topology on $\Sub({\rm L}(\Gamma))$ is the so-called Effros--Maréchal topology, which was introduced by Effros \cite{Effros} in 1965 as a Borel structure, until  Maréchal \cite{Marechal} proved in 1973 that it actually gives rise to a Polish topology. Further characterisations have been provided by Haagerup--Winsl{\o}w \cite{haagerup1}. The original Maréchal's proof of Polishness has been corrected in \cite{fima2024michael} and a new point of view on this topology, via UCP maps and conditional expectations, has also been developed in \cite{FimaLeMaitreMukherjeePatri2026}. Amrutam, Hartman and the second-named author \cite{AHO25} introduced the concept of \textit{invariant random von Neumann subalgebras} (short: IRA or von Neumann IRA) as $\Gamma$-invariant Borel probability measures on $\Sub({\rm L}(\Gamma))$. 

\vspace{0.2cm}

\noindent In the present work, we consider the Wijsman topology, which appeared already in Hausdorff \cite{Hausdorff} in 1914, and later in Frolík \cite{Frolik} and Wijsman \cite{Wijsman}. On the space $\Sub(A)$ of unital $C^*$-subalgebras of a given unital $C^*$-algebra $A$, the Wijsman topology is Polish whenever $A$ is separable (Theorem \ref{thm:polish}). Our first main result shows that, in contrast with the space of subgroups, the space of $C^*$-subalgebras is rarely compact.

\begin{thmx}[Theorem~\ref{thm:local-compactness}]\label{introthm: compact}
Let $A$ be a unital $C^*$-algebra. The following are equivalent:
\begin{enumerate}
\item $A$ is finite-dimensional;
\item $\Sub(A)$ is compact;
\item $\Sub(A)$ is locally compact.
\end{enumerate}
\end{thmx}

\noindent We then characterise exactness as the limit of nuclear algebras in the Wijsman sense as follows.

\begin{thmx}[Corollary~\ref{cor:exact-nuclear-wijsman-approximation}]\label{introthm: exactness}
A separable unital $C^*$-algebra $B$ is exact if and only if 
there exist a separable unital $C^*$-algebra $A\supseteq B$ and
nuclear $B_n\in\Sub(A)$ such that $B_n\to B$ in the Wijsman topology.
\end{thmx}

 \noindent Let us now focus on the reduced $C^*$-algebra $C_r^*(\Gamma)$ of a countable discrete group $\Gamma$, viewed with the conjugation action $\Gamma \curvearrowright \Sub(C_r^*(\Gamma))$. We define a \textbf{random subalgebra} (short RA) of $C^*_r(\Gamma)$ as a Borel probability measure on $\Sub(C_r^*(\Gamma))$. We study three categories of these measures: 
\begin{itemize}
    \item quasi-invariant (qIRA), 
    \item $\mu$-stationary ($\mu$-SRA) for a probability measure $\mu$ on $\Gamma$, 
    \item invariant measures (IRA or $C^*$-IRA).
\end{itemize}

\noindent First, we relate our notion to the well-studied group setup. In particular we show in Theorem~\ref{thm:reduced-group-map} that the map $\Sub(\Gamma)\to\Sub(C^*_r(\Gamma))$, $\Lambda\mapsto C^*_r(\Lambda)$ is Borel, injective and $\Gamma$-equivariant. Hence, every IRS can be pushed forward to a $C^*$-IRA.

\vspace{0.2cm}

\noindent Also, we show in Corollary~\ref{cor:bi-commutant-map-borel} that the map $\Sub(C^*_r(\Gamma))\to\Sub({\rm L}(\Gamma))$, $B\mapsto B''$ is Borel and $\Gamma$-equivariant so every $C^*$-IRA can be push forward to a von Neumann IRA of ${\rm L}(\Gamma)$ (in a non-injective way in general). The same holds for the other two categories of measures.
$C^*$-RAs which do not come from groups we call \textbf{exotic}. Von Neumann IRAs that do not arise from $C^*$-algebras we call \textbf{super-exotic}. In Section~\ref{sec:exotic} we provide characterisations and examples of groups with exotic and super-exotic qIRA and IRA, respectively. In particular, the reduced $C^*$-algebra of the  ''traffic'' lamplighter group $\mathbb{Z}\wr (\mathbb{Z}/3\mathbb{Z})$ hosts exotic IRAs. On ${\rm L}(\mathbb{Z}\wr \mathbb{Z})$ there are super-exotic IRAs (Corollary~\ref{cor:nonatomic-exotic}). In Section~\ref{sec:noexotic}, we provide examples of groups $\Gamma$ for which $C^*_r(\Gamma)$ has no exotic IRA.

\begin{thmx}[Corollary~\ref{cor:hyperbolic_no_exotic}]\label{thm:hyperbolic_no_exotic}
If $\Gamma$ is a torsion-free non-amenable hyperbolic group, then $C_r^*(\Gamma)$ has no exotic IRA.
\end{thmx}

\noindent Theorem \ref{thm:hyperbolic_no_exotic} generalises \cite[Theorem~4.4]{AmrutamJiang2026}, which asserts that every unital $\Gamma$-invariant $C^*$-subalgebra of $C_r^*(\Gamma)$ is $C_r^*(N)$ for some normal subgroup $N\triangleleft\Gamma$, by replacing invariance of the subalgebra with invariance of its distribution. The deterministic case of \cite{AmrutamJiang2026} can be deduced from Theorem \ref{thm:hyperbolic_no_exotic} by considering Dirac measures.

\vspace{0.2cm}

\noindent Our main applications of random subalgebras concern nuclearity and simplicity.
Hartman--Kalantar \cite{hartman2023stationary} characterised $C^*$-simplicity of $\Gamma$ in terms of unique stationary states. With the notion of probability measures on $\Sub(C_r^*(\Gamma))$, we can provide a probabilistic analogue as follows.  A probability measure is called \textbf{nuclear} if the set of nuclear subalgebras is a co-null set.

\begin{thmx}[Theorem~\ref{thm:simple}]\label{introthm:simple} 
$\Gamma$ is $C^*$-simple if and only if there exists
$ \mu\in\Prob(\Gamma)$ such that $\delta_{\mathbb C1}$ is  the unique nuclear $\mu$-SRA. 
\end{thmx}

\noindent Given an IRA $\nu$ of $C_r^*(\Gamma)$, we show that then there exists a unique maximal $C^*$-subalgebra $A\in\Sub(C_r^*(\Gamma))$ such that $\nu(\Sub(\Gamma))=1$ (Proposition~\ref{prop:alg sup unique}). We denote it by $\langle \nu \rangle$. If $C_r^*(\Gamma)=\langle \nu \rangle$, then we call $\nu$ a \textbf{spanning} IRA.

\vspace{0.2cm}

\noindent It was shown by Bader--Duchesne--Lécureux \cite{BDL} that every amenable IRS is supported on the amenable radical $\Rad(\Gamma)$. This was extended to group von Neumann algebras in \cite{AHO25} and to cross-products by \cite{Zhou}. We provide the result for reduced $C^*$-algebras.
\begin{thmx}[Corollary \ref{thm:maximal-nuclear-invariant}, Theorem~\ref{thm: nuclear IRA}, Corollary \ref{cor:nuclear-ira-regular}]\label{introthm:nuclear IRA} 
Let $\Gamma$ be a countable discrete group. The following holds.
\begin{enumerate}
    \item The unique maximal $\Gamma$-invariant unital nuclear $C^*$-subalgebra of $C_r^*(\Gamma)$ is $C_r^*(\Rad(\Gamma))$.
    \item If $\nu$ is a nuclear IRA of $C_r^*(\Gamma)$, then $\langle \nu \rangle$ lies inside $C_r^*(\Rad(\Gamma))$. In particular, if $\nu$ is spanning and nuclear, then $\Gamma$ is amenable.
    \item The amenable radical is trivial if and only if the only nuclear $\IRA$ of $C_r^*(\Gamma)$ is  $\delta_{\C1}$.
\end{enumerate}
\end{thmx}

\subsection*{Organisation of the paper}
Section~\ref{sec: top} is devoted to topological results, including the compactness characterisation of $\Sub(A)$ in Theorem~\ref{introthm: compact} and topological properties of the maps from subgroups to $C^*$-subalgebras and from $C^*$-subalgebras to von Neumann subalgebras are provided in Theorem~\ref{thm:reduced-group-map} and Corollary~\ref{cor:bi-commutant-map-borel}, respectively.
In Section~\ref{sec: nuc} we show that the space of nuclear $C^*$-subalgebras is a $G_{\delta}$-set  and prove the exactness characterisation of Theorem~\ref{introthm: exactness}. The introduction of probability measures on the space $\Sub(A)$ is provided in Section~\ref{sec:random-rigidity}, which also contains the simplicity characterisation of Theorem~\ref{introthm:simple} and the nuclearity statement of Theorem~\ref{introthm:nuclear IRA}. In Section~\ref{sec:exotic}, we investigate the existence of exotic and super-exotic IRAs. Finally, Section~\ref{sec:noexotic} provides the rigidity result of Theorem~\ref{thm:hyperbolic_no_exotic}.
\subsection*{AI statement} No AI was used for this work.

\section{\texorpdfstring{The Wijsman space of $C^*$-subalgebras}{The Wijsman space of C*-subalgebras}}\label{sec: top}

\subsection{The Wijsman and lower topologies} Let $(X,d)$ be a metric space and let $\mathcal F^*(X)$ be the set of its
non-empty closed subsets. The \textbf{Wijsman topology} on $\mathcal F^*(X)$ is the
coarsest topology for which every map $F\longmapsto d(x,F)=\inf_{y\in F}d(x,y)$, $x\in X$,
is continuous.

\medskip

\noindent The \textbf{lower topology} on $\mathcal F^*(X)$ is generated by the sets
$$
\mathcal I_U=\{F\in\mathcal F^*(X):F\cap U\neq\emptyset\},
$$
where $U$ ranges over the open subsets of $X$. Equivalently, it is the
coarsest topology for which every map $F\mapsto d(x,F)$ is upper
semicontinuous \cite[Proposition~3.6]{fima2024michael}.
In particular, the lower topology is coarser than the
Wijsman topology: each distance map is Wijsman-continuous, hence upper
semicontinuous. Thus every lower-open set is Wijsman-open. The same
comparison holds for the topologies induced on any subset of
$\mathcal F^*(X)$.

\medskip

\noindent Following the convention of \cite[Definition~24.1]{Kechris1995},
a map $f:X\to Y$ between topological spaces is of \textbf{Baire class~1} if $f^{-1}(U)$ is
$F_\sigma$ for every open $U\subseteq Y$, or equivalently if the inverse
image of every closed set is $G_\delta$. Here $F_\sigma$ means a countable
union of closed sets, and $G_\delta$ a countable intersection of open sets.

\smallskip
\noindent The next Lemma is a special case of \cite[Theorem~5.1(b)]{beerPolishtopologyclosed1991}, when $X$ is supposed to be metric.
\begin{lemma}\label{lem:lower-to-borel}
Let $X$ be a topological space, $Y$ be a separable metric space, and
$f:X\to\mathcal F^*(Y)$ be a map. If $f$ is continuous for the lower topology, then $f$ is
Borel when $\mathcal F^*(Y)$ is endowed with the Wijsman topology.
If every open subset of $X$ is $F_\sigma$, in particular if $X$ is
metrizable, then $f$ is of Baire class~1.
\end{lemma}

\begin{proof}
Fix a countable dense set $D\subseteq Y$. Since
$|d(y,F)-d(z,F)|\leq d(y,z)$, the maps $F\mapsto d(y,F)$,
$y\in D$, generate the Wijsman topology.
Lower continuity of $f$ makes $u_y(x):=d(y,f(x))$ upper semicontinuous since
$$\{x\in X:u_y(x)<r\}=\{x\in X:f(x)\cap B(y,r)\neq\emptyset\}\quad\text{for all}\quad r>0.$$
Hence $\{x\in X:u_y(x)<r\}$ is open and
$$
\{x\in X:u_y(x)>r\}
=\bigcup_{n\geq1}\{x\in X:u_y(x)\geq r+1/n\}\text{ is }F_\sigma\text{ for all }r>0.
$$
These sets, for $y\in D$ and $r\in\mathbb Q$, are the inverse images under $f$ of a
countable subbasis for the Wijsman topology on $\mathcal F^*(Y)$. The Borel conclusion follows from
\cite[Section~11.C, p.~71]{Kechris1995}. If every open subset of $X$ is
$F_\sigma$, all these inverse images are $F_\sigma$; closure under finite
intersections and countable unions gives the second conclusion.
\end{proof}

\begin{remark}\label{rem:lower-wijsman-borel}
Let $Y$ be a separable metric space and $D\subseteq Y$ a countable dense set.
The Wijsman and lower topologies on $\mathcal F^*(Y)$ generate the same
Borel $\sigma$-algebra, namely the Effros $\sigma$-algebra, generated by the sets
$$
\mathcal I_{B(y,r)}
=\{F\in\mathcal F^*(Y):d(y,F)<r\},
\qquad y\in D,\quad r\in\mathbb Q,\quad r>0.
$$
This identification is attributed to Hess~\cite{Hess1983} in
\cite[Section~3]{beerPolishtopologyclosed1991}.
The reduction to a countable basis is recalled in
\cite[Section~12.C]{Kechris1995}.
Indeed, these sets form a countable subbasis for the lower topology, so
its open sets are countable unions of finite intersections of these sets.
Applying Lemma~\ref{lem:lower-to-borel} to the identity map on
$\mathcal F^*(Y)$ with the lower topology on the domain shows that every
Wijsman Borel set is lower Borel. The converse follows because the lower
topology is coarser than the Wijsman topology.
The two induced topologies on any subset of $\mathcal F^*(Y)$ likewise
generate the same Borel $\sigma$-algebra, generated by the traces of the
sets $\mathcal I_{B(y,r)}$; see also \cite[Section~2]{BarbatiBeerHess1994}.
\end{remark}

\subsection{\texorpdfstring{The space of $C^*$-subalgebras}{The space of C*-subalgebras}}
If $A$ is a unital $C^*$-algebra, set
$$
\Sub(A)=\{B\subseteq A:B\text{ is a unital }C^*\text{-subalgebra}{,\ 1_A\in B}\}
\subseteq\mathcal F^*(A),
$$
and endow $\Sub(A)$ with the induced Wijsman topology.

\medskip

\noindent The following fact was observed in
\cite[Remark~4.9]{fima2024michael}; we include a proof for completeness.

\begin{theorem}\label{thm:polish}
Let $A$ be a separable unital $C^*$-algebra. Then $\Sub(A)$ is a closed
subspace of $\mathcal F^*(A)$ and is Polish for the Wijsman topology.
\end{theorem}

\begin{proof}
By \cite[Thm.~4.3]{beerPolishtopologyclosed1991}, $\mathcal F^*(A)$ is
Polish. It therefore suffices to prove that $\Sub(A)$ is sequentially
closed. Let $(B_n)$ be a sequence in $\Sub(A)$ converging in the Wijsman
topology to $F\in\mathcal F^*(A)$. Since $d(0,F)=\lim_n d(0,B_n)=0$ and $d(1,F)=\lim_n d(1,B_n)=0$,
one has $0,1\in F$. Fix $x,y\in F$ and $\lambda\in\C$. Choose
$x_n,y_n\in B_n$ such that $x_n\to x$ and $y_n\to y$. Then
$$
x_n+y_n\to x+y,\quad \lambda x_n\to\lambda x,
\quad x_n^*\to x^*,\quad x_ny_n\to xy.
$$
For each of the four limits $z$ above, there exists a sequence $z_n\in B_n$
converging to $z$. Hence $d(z,B_n)\to0$, and Wijsman convergence gives
$d(z,F)=0$. Since $F$ is closed, $z\in F$. Thus $F$ is a unital
$C^*$-subalgebra of $A$.
\end{proof}

\begin{proposition}\label{rem:subspace-closed}
Let $B\in\Sub(A)$ and let $j:\Sub(B)\to\Sub(A)$ be the inclusion.
\begin{enumerate}
\item $j(\Sub(B))$ is Wijsman-closed.
\item $j$ is continuous for the lower topologies. If $A$ is separable, it is Borel from the lower topology to the Wijsman topology, and of Baire class~$1$ for the Wijsman topologies.
\item $j$ is Wijsman-continuous at $B$.
\item The inverse $j^{-1}:j(\Sub(B))\to\Sub(B)$ is continuous from the induced Wijsman topology to the intrinsic Wijsman topology.
\end{enumerate}
\end{proposition}
\begin{proof}
$(1)$. If $C\nsubseteq B$, choose $c\in C\setminus B$. The open set
$$\{C'\in\Sub(A):d(c,C')<d(c,B)\}$$ contains $C$ and has empty intersection with $j(\Sub(B))$.

\medskip

\noindent$(2)$. For every norm-open $U\subseteq A$, one has $
j^{-1}(\mathcal I_U\cap\Sub(A))=\mathcal I_{U\cap B}\cap\Sub(B)$.
This proves lower continuity. If $A$ is separable, Lemma~\ref{lem:lower-to-borel}, applied with either the lower source topology, shows that $j$ is Borel. Since the intrinsic Wijsman source topology is Polish, by Theorem~\ref{thm:polish}, it finishes the proof of $(2)$.

\medskip

\noindent$(3)$. For a net $C_i\to B$ in $\Sub(B)$, $x\in A$ and $b\in B$,
$$
d(x,B)\leq d(x,C_i)\leq\|x-b\|+d(b,C_i).
$$
Taking upper limits and then the infimum over $b\in B$ finishes the proof.

\medskip

\noindent$(4)$. For $b\in B$ and $C\in j(\Sub(B))$, one has $d_B(b,j^{-1}(C))=d_A(b,C)$.\end{proof}

\begin{example}The topology on $\Sub(B)$ induced from the Wijsman topology on $\Sub(A)$ can be
strictly finer than the intrinsic Wijsman topology of $\Sub(B)$, even
when $A$ and $B$ are commutative and separable.

\medskip

\noindent Indeed, let $\mathbb N_0=\{0,1,2,\ldots\}$ and $B=\{b\in\ell^\infty(\mathbb N_0):\lim_{k\to\infty}b_k
       \text{ exists in }\mathbb C\}\subset\ell^\infty(\mathbb N_0)$ be the $C^*$-subalgebra of convergent complex sequences. Consider the element $a=(2,1,0,1,0,\ldots)\in\ell^\infty(\mathbb N_0)$ and define $A:=C^*(B,a)\subseteq\ell^\infty(\mathbb N_0)$. $A$ is commutative and separable and $B\in\Sub(A)$. For $n\geq1$, let
$$C_n=\{b\in B:b_0=b_{2n-1}\}\quad\text{and}\quad C=\{b\in B:b_0=\lim_{k\to\infty}b_k\}.$$
These are unital $C^*$-subalgebras of $B$ and:
$$
 d(b,C_n)=\frac{|b_0-b_{2n-1}|}{2}
 \longrightarrow\frac{|b_0-L|}{2}=d(b,C).
$$
Indeed, for $b\in B$ with limit $L$, replacing $b_0,b_{2n-1}$ by their midpoint gives the first distance equality above. For the second, every $c\in C$ satisfies
$|b_0-L|\leq |b_0-c_0|+|c_0-L|\leq2\|b-c\|$.
Conversely, fix $\varepsilon>0$, put $m=(b_0+L)/2$, and replace $b_0$ and all sufficiently late coordinates by $m$. The resulting element of $C$ has distance at most $|b_0-L|/2+\varepsilon$ from $b$.

\medskip

\noindent For $a$ one has $d(a,C_n)\geq1/2$ and equality is attained by the sequence with coordinates $c_0=c_{2n-1}=3/2$ and $c_k=1/2$ otherwise. For $c\in C$, the values $a_0=2$ and $a_{2k}=0$ give $\|a-c\|\geq1$, and the constant sequence $1$ attains this bound. Hence $d(a,C_n)=1/2$ and $d(a,C)=1$. Hence $C_n\to C$ for the intrinsic Wijsman topology on $\Sub(B)$ but $C_n$ does not converge to $C$ in $\Sub(A)$.
\end{example}

\subsection{Local compactness}
\noindent For existence and uniqueness in the next lemma, see \cite[Lemma~2.2.7]{BekkaHarpeValette2008}. We include an easy proof for the reader's convenience.
\begin{lemma}\label{lem:minimal-enclosing-disk}
Let $K\subseteq\C$ be a non-empty compact set. There exists a unique pair
$(x,r)\in\C\times[0,\infty)$ such that $K\subseteq\overline{B(x,r)}$
and, for every $(x',r')\in\C\times[0,\infty)$ satisfying
$K\subseteq\overline{B(x',r')}$, one has $r\leq r'$.
We denote $x$ and $r$ by $\mathrm{Center}(K)$ and
$\mathrm{Radius}(K)$, respectively.
{For non-empty compact sets $K,L\subseteq\C$,
$$
|\mathrm{Radius}(K)-\mathrm{Radius}(L)|\leq d_{\mathcal{H}}(K,L),
$$
where $d_{\mathcal{H}}(K,L)=\max\{\sup_{z\in K}d(z,L),\sup_{w\in L}d(w,K)\}$ is the Hausdorff distance.}
\end{lemma}

\begin{proof}

Define $f(y)=\max_{z\in K}|z-y|$ for $y\in\C$.
The triangle inequality gives $|f(y)-f(y')|\leq|y-y'|$, so $f$ is continuous.
Fix $z_*\in K$. Since $f(y)\geq|y-z_*|\to\infty$ as $|y|\to\infty$,
$f$ attains its minimum, which we denote by $r$.

If $f(y_1)=f(y_2)=r$, then, for every $z\in K$,
$$
\left|z-\frac{y_1+y_2}{2}\right|^2
=\frac{|z-y_1|^2+|z-y_2|^2}{2}-\frac{|y_1-y_2|^2}{4}
\leq r^2-\frac{|y_1-y_2|^2}{4}.
$$
Taking the maximum over $z\in K$ gives $r^2\leq r^2-|y_1-y_2|^2/4$ by minimality of $r$, hence $y_1=y_2$.
Thus $f$ has a unique minimiser $x$.
Finally, $K\subseteq\overline{B(x',r')}$ is equivalent to $f(x')\leq r'$,
so $(x,r)$ is the unique pair with the stated property.
For every $y\in\C$,
$$
\max_{z\in K}|z-y|\leq\max_{w\in L}|w-y|+d_{\mathcal{H}}(K,L).
$$
Taking the infimum over $y$ and interchanging $K$ and $L$ proves the last assertion.
\end{proof}

\begin{proposition}\cite[Proposition~9.9]{Pelczynski1968}\label{prop:rad-dist-formula}
Let $X$ be a non-empty compact Hausdorff space and $R$ a closed equivalence relation on $X$. Identify $C(X/R)$ with its pullback under $p_R:X\to X/R$. Then, for every $f\in C(X)$,
$$d(f,C(X/R))=\sup_{y\in X/R}\mathrm{Radius}\bigl(f(p_R^{-1}(\{y\}))\bigr).$$
\end{proposition}

\begin{theorem}\label{thm:local-compactness}
Let $A$ be a unital $C^*$-algebra. The following are equivalent:
\begin{enumerate}
\item $A$ is finite-dimensional;
\item $\Sub(A)$ is compact;
\item $\Sub(A)$ is locally compact.
\end{enumerate}
\end{theorem}

\begin{proof}
Suppose first that $A$ is finite-dimensional. Let $(B_n)$ be a sequence
in $\Sub(A)$ and put $f_n(x)=d(x,B_n)$. Since the functions $f_n$ are
$1$-Lipschitz and $0\leq f_n(x)\leq\|x\|$, diagonal extraction on a
countable dense subset of $A$, followed by the Lipschitz estimate, yields
a subsequence converging pointwise on $A$ to a $1$-Lipschitz function $f$.
Relabel this subsequence and put $B=f^{-1}(0)$. Then $B$ is closed,
$0,1\in B$, and $f(x)\leq d(x,B)$ for every $x\in A$. Let us show the converse inequality. Fix $x\in A$ and choose $b_n\in B_n$ with $\|x-b_n\|\leq f_n(x)+1/n$.
The sequence $(b_n)$ is bounded. Since $A$ is finite dimensional, a subsequence $(b_{n_k})$ converges
to some $b\in A$. Note that
$$
0\leq f_{n_k}(b)\leq\|b-b_{n_k}\|\longrightarrow0.
$$
Hence $f(b)=0$ so $b\in B$ and,
$$
d(x,B)\leq\|x-b\|=\lim_k\|x-b_{n_k}\|
=\lim_k f_{n_k}(x)=f(x).
$$
Thus $f=d(\,\cdot\,,B)$ and $B_n\to B$ in the Wijsman topology.
{By Theorem~\ref{thm:polish}, $B\in\Sub(A)$.
Thus $\Sub(A)$ is sequentially compact; since it is metrizable by the same theorem, it is compact.
This proves $(1)\Rightarrow(2)$.} The implication
$(2)\Rightarrow(3)$ is immediate.

\medskip

\noindent For $(3)\Rightarrow(1)$, suppose that $A$ is infinite-dimensional.
By {\cite[Remark, p.~4]{tullo1976conditions}}, there is a self-adjoint $a\in A$
with infinite spectrum $X$. Put $B=C^*(1,a)\simeq C(X)$; the space $X$
is compact metrizable. We show {below} that $\Sub(A)$ has no compact neighbourhood
of the point $B$.

\medskip

\noindent By Proposition~\ref{rem:subspace-closed}, the inclusion $j:\Sub(B)\to\Sub(A)$ is continuous at $B$ and the image of
$j$ is closed and its inverse is continuous. Consequently, if $K$ were
a compact neighbourhood of $B$ in $\Sub(A)$, then $K_B=j^{-1}(K\cap j(\Sub(B)))$
would be {a compact neighbourhood} of $B$ in the intrinsic Wijsman topology. Identifying $B$ with $C(X)$, choose a basic neighbourhood
$$V=\{C\in\Sub(C(X)):d(f_i,C)<\varepsilon,\ 1\leq i\leq m\}
\subseteq K_B,$$
where $f_1,\ldots,f_m\in C(X)$ and $\varepsilon>0$. Since $X$ is compact
and infinite, it has a non-isolated point $q$. For $f\in C(X)$ and
nonempty $U\subseteq X$, put
$$
\operatorname{osc}_U(f)=\sup_{s,t\in U}|f(s)-f(t)|.
$$
{By continuity, choose open neighbourhoods $U_i$ of $q$ with
$|f_i(t)-f_i(q)|<\varepsilon/3$ for $t\in U_i$.
For $U=\bigcap_{i=1}^mU_i$ and $1\leq i\leq m$,
$\operatorname{osc}_U(f_i)
\leq 2\sup_{t\in U}|f_i(t)-f_i(q)|\leq 2\varepsilon/3<\varepsilon$.}
Since every neighbourhood of $q$ is infinite, we can choose distinct $p,r\in U\setminus\{q\}$ and a sequence of
distinct points $y_n\in U\setminus\{p,q,r\}$ converging to $q$. Set
$$
C_n=\{f\in C(X):f(p)=f(q),\ f(r)=f(y_n)\}{\in\Sub(C(X))}.
$$
Note that $C_n=C(X/R_n)$, where $R_n$ is the closed equivalence relation on $X$ with the only non-singleton equivalence classes being $\{p,q\}$ and $\{r,y_n\}$, so Proposition~\ref{prop:rad-dist-formula}, together with $\mathrm{Radius}(\{u,v\})=|u-v|/2$, gives
$$
d(f,C_n)=\max\left\{\frac{|f(p)-f(q)|}{2},
                         \frac{|f(r)-f(y_n)|}{2}\right\}.
$$
In particular $C_n\in V\subseteq K_B$ for every $n$, and
$$
d(f,C_n)\longrightarrow D_f:=
\max\left\{\frac{|f(p)-f(q)|}{2},\frac{|f(r)-f(q)|}{2}\right\}
\quad\text{for all }f\in C(X).
$$
Since $K_B$ is compact and $\Sub(C(X))$ is metrizable, a subsequence
would converge to some $C\in K_B$. Necessarily
$$
C=\{f:D_f=0\}=\{f\in C(X):f(p)=f(q)=f(r)\}.
$$
By the Tietze extension theorem, choose $h\in C(X,[-1,1])$ with
$h(p)=-1$, $h(q)=0$ and $h(r)=1$. Then $D_h=1/2$, whereas
Proposition~\ref{prop:rad-dist-formula} gives
$$
d(h,C)=\mathrm{Radius}(\{-1,0,1\})=1\ne D_h,
$$
contradicting Wijsman convergence.
Hence no such compact neighbourhood $K$
exists, and $\Sub(A)$ is not locally compact.
\end{proof}

\subsection{Homomorphisms and approximation by conjugates}

\begin{lemma}\label{lem:direct-image}
Let $\pi:A\to D$ be a unital $*$-homomorphism between unital $C^*$-algebras. The map
$\pi_{\mathrm{im}}:\Sub(A)\to\Sub(D)$, $B\mapsto\pi(B)$
is continuous for the lower topologies. If $D$ is separable, it is Borel from the lower topology to the Wijsman topology. If $A$ and $D$ are separable, it is of Baire class~$1$ for the Wijsman topologies.
\end{lemma}
\begin{proof}
For every norm-open $U\subseteq D$,
$\pi_{\mathrm{im}}^{-1}(\mathcal I_U\cap\Sub(D))
=\mathcal I_{\pi^{-1}(U)}\cap\Sub(A)$
so $\pi_{\mathrm{im}}$ is lower-continuous. The remaining assertions follow from Lemma~\ref{lem:lower-to-borel}, using Theorem~\ref{thm:polish} for metrizability when $A$ is separable.
\end{proof}

\begin{lemma}\label{lem:inverse-image}
Let $\pi:A\to B$ be a surjective unital $*$-homomorphism between unital
$C^*$-algebras. Then $\pi^{-1}:\Sub(B)\to\Sub(A)$, $D\longmapsto\pi^{-1}(D)$, is a homeomorphism onto the closed subset
$\{C\in\Sub(A):\ker\pi\subseteq C\}$ of $\Sub(A)$.
\end{lemma}

\begin{proof}
We first prove that, for $a\in A$ and $D\in\Sub(B)$, one has:
$$d_A(a,\pi^{-1}(D))=d_B(\pi(a),D).$$
Let $b\in\pi^{-1}(D)$ and write $d=\pi(b)\in D$. Contractivity of $\pi$ gives
$$d_B(\pi(a),D)\leq\|\pi(a)-d\|\leq\|a-b\|.$$
Taking the infimum over $b$ proves $d_B(\pi(a),D)\leq d_A(a,\pi^{-1}(D))$. For the converse inequality, fix $d\in D$ and $\varepsilon>0$. By surjectivity of $\pi$ and the quotient norm formula, there exists $z\in A$ such that $\pi(z)=\pi(a)-d$ and  $\|z\|\leq\|\pi(a)-d\|+\varepsilon$. Then $a-z\in\pi^{-1}(D)$ so $d_A(a,\pi^{-1}(D))\leq\|z\|
\leq\|\pi(a)-d\|+\varepsilon$. Taking the infimum over $d\in D$ and then letting $\varepsilon$ tend to
zero proves the converse inequality.

\medskip

\noindent Surjectivity gives $\pi(\pi^{-1}(D))=D$, so $D\mapsto\pi^{-1}(D)$ is injective.
Every $\pi^{-1}(D)$ contains $\ker\pi$. Conversely, if $C\in\Sub(A)$ contains
$\ker\pi$, then $\pi(C)\in\Sub(B)$, since images of $C^*$-algebras under
$*$-homomorphisms are closed, and $\pi^{-1}(\pi(C))=C$:
if $\pi(a)=\pi(c)$ with $c\in C$, then $a-c\in\ker\pi\subseteq C$, hence $a\in C$.
Thus the image is $\{C:\ker\pi\subseteq C\}$, which is closed since it equals
$\bigcap_{x\in\ker\pi}\{C:d_A(x,C)=0\}$.
The distance identity and surjectivity of $\pi$ give continuity in both
directions.
\end{proof}

\begin{proposition}\label{rem:universal-conjugate-approximation}
For every unital inclusion $B\subseteq D$ of unital separable $C^*$-algebras,
there are a separable unital $C^*$-algebra $A\supseteq D$ and
self-adjoint unitaries $u_n\in A\cap B'$ such that $u_nDu_n\to B$
in the Wijsman topology on $\Sub(A)$.
\end{proposition}

\begin{proof}
Let $A=D*_B\bigl(*_{n\geq1,B}(B\oplus B)\bigr)$ be the full amalgamated free product, with diagonal embeddings of $B$ into each $B\oplus B$. By letting $u_n=(1,-1)$ in the $n$th copy of $B\oplus B$, we may view $A$ as the universal unital $C^*$-algebra generated by $D$ and
self-adjoint unitaries $u_n$, $n\geq 1$ subject to the relations $u_nb=bu_n$ for all $b\in B$ and $n\geq 1$. Sending every $u_n$ to $1$ gives a retraction onto $D$, so $D$ embeds
faithfully in $A$. Moreover, $A$ is separable. Put
$B_n=u_nDu_n$ and $A_N=C^*(D,u_1,\ldots,u_N)$. We claim that $ d(y,B_n)=d(y,B)$ for all $y\in A_N$ and  $n>N$. Indeed, fix $y\in A_N$ and $n>N$. By \cite[Theorem~3.2]{Rieffel2011}, there are
a unital representation $\pi:A_N\to\mathcal B(H)$ and a self-adjoint
unitary $U\in\pi(B)'$ such that $d(y,B)=\frac12\|\pi(y)-U\pi(y)U\|$. By the universal property of $A$, there are unital $*$-homomorphisms
$\rho_0,\rho_1:A\to\mathcal B(H)$ satisfying $\rho_0|_{A_N}=\pi$, $\rho_1|_{A_N}=U\pi(\cdot)U$ and:
$$
\rho_0(u_j)=1\text{ for all }j>N,\quad\rho_1(u_j)=\left\{\begin{array}{ll}U&\text{ if }j=n,\\1&\text{ if }j>N\text{ and }j\neq n.\end{array}\right.
$$
These assignments satisfy the relations since $U\in\pi(B)'$. For every $x\in D$,
$$
\rho_0(u_nxu_n)=\pi(x)
=U(U\pi(x)U)U=\rho_1(u_nxu_n),
$$
so $\rho_0$ and $\rho_1$ agree on $B_n$. Hence, for every $z\in B_n$,
$$
d(y,B)=\frac12\|\rho_0(y-z)-\rho_1(y-z)\|\leq\|y-z\|.
$$
Taking the infimum over $z\in B_n$ gives $d(y,B)\leq d(y,B_n)$. The reverse inequality follows from $B\subseteq B_n$. This proves our claim.

\medskip

\noindent Since $\bigcup_N A_N$ is dense in $A$ and distance functions are
$1$-Lipschitz, it follows that $B_n\to B$ in the Wijsman topology.\end{proof}

\begin{remark}\label{rem:direct-image-discontinuous}
The map in Lemma~\ref{lem:direct-image} need not be continuous for the Wijsman topologies, even for a surjective unital $*$-homomorphism admitting a unital $*$-homomorphic section. Indeed, for a proper unital inclusion $B\subsetneq D$ of separable $C^*$-algebras, take $A$ and $(u_n)$ from Proposition~\ref{rem:universal-conjugate-approximation}. The retraction $\pi:A\to D$ fixing $D$ and sending every $u_n$ to $1$ satisfies
$$
B_n:=u_nDu_n\longrightarrow B,\qquad
\pi_{\mathrm{im}}(B_n)=D\neq B=\pi_{\mathrm{im}}(B).
$$
Since $\Sub(D)$ is Hausdorff, $\pi_{\mathrm{im}}$ is not continuous at $B$. In contrast, its restriction to $\{C\in\Sub(A):\ker\pi\subseteq C\}$ is a homeomorphism onto $\Sub(D)$, with inverse $E\mapsto\pi^{-1}(E)$, by Lemma~\ref{lem:inverse-image}.
\end{remark}

\subsection{The space of subgroups}

Let $\Gamma$ be a countable discrete group. Let $\Sub(\Gamma)$ denote the space of subgroups of $\Gamma$ equipped with the \textbf{Chabauty topology}, i.e. the induced product topology coming from $\{0,1\}^{\Gamma}$.  The full $C^*$-algebra is denoted by $C^*(\Gamma)$. Given a unitary representation $\pi$ of $\Gamma$ we use the notation $C^*_\pi(\Gamma):=\pi(C^*(\Gamma))$. In the case $\pi=\lambda_\Gamma$ is the left regular representation, we obtain the reduced $C^*$-algebra, denoted by $C^*_r(\Gamma)$.

\begin{theorem}\label{thm:reduced-group-map}
Let $\pi$ be a unitary representation of $\Gamma$ and consider the map
$$
\iota_\pi:\Sub(\Gamma)\to\Sub(C_\pi^*(\Gamma)),\quad H\mapsto C^*(\pi(h):h\in H).
$$
The following holds.
\begin{enumerate}
\item The map $\iota_\pi$ is continuous from the Chabauty topology to
the lower topology, and is of Baire class~$1$ for the Wijsman target topology.
\item If $\lambda_\Gamma\prec\pi$, then $\iota_\pi$ is injective and
its inverse on its image is continuous for the induced Wijsman topology.
\item For the regular and universal representations, the images
$I_\Gamma$ and $I_\Gamma^{\rm full}$ of respectively 
$\iota_r:H\mapsto C_r^*(H)$ and $\iota:H\mapsto C^*(H)$ are Wijsman-closed.
\end{enumerate}
\end{theorem}

\begin{proof}
Put $A=C_\pi^*(\Gamma)$, $v_g=\pi(g)$, and $A_H=C^*(v_h:h\in H)$. Suppose $H_n\to H$. Given $x\in A$ and
$\varepsilon>0$, choose $y\in\operatorname{span}\{v_h:h\in H\}$ with
$\|x-y\|<d(x,A_H)+\varepsilon$. Its finite support is eventually
contained in $H_n$, so
$$
\limsup_n d(x,A_{H_n})\leq\|x-y\|<d(x,A_H)+\varepsilon.
$$
Hence, distance maps are upper semicontinuous, proving lower-continuity in $(1)$, and Lemma~\ref{lem:lower-to-borel} gives Baire class~$1$.

\medskip

\noindent For $(2)$, weak containment gives a quotient
$q:A\to C_r^*(\Gamma)$ with $q(v_g)=\lambda_g$ and $q(A_H)=C_r^*(H)$.
If $g\notin H$ and $b\in C_r^*(H)$, then $b\delta_e\in\ell^2(H)$ and
$$
\|\lambda_g-b\|\geq\|\delta_g-b\delta_e\|_2
=\sqrt{1+\|b\delta_e\|_2^2}\geq1.
$$
Contractivity of $q$, together with $0\in A_H$ and $\|v_g\|=1$, gives
$$
d(v_g,A_H)=\mathbf1_{\Gamma\setminus H}(g)
\quad\text{for all}\quad g\in\Gamma,\ H\leq\Gamma.
$$
This proves injectivity and continuity of the inverse.

\medskip

\noindent For $(3)$, take $A=C_r^*(\Gamma)$ or $C^*(\Gamma)$ and use the canonical
conditional expectations
$$
E_H:A\longrightarrow A_H\quad E_H(v_g)=\mathbf1_H(g)v_g.
$$
Suppose $A_{H_n}\to B$ in $\Sub(A)$. For every $g\in\Gamma$,
$$
1_{H_n}(g)=1-d(v_g,A_{H_n})\rightarrow1-d(v_g,B)\in\{0,1\}.
$$
The limit equals $1$ exactly when $v_g\in B$. Thus $H_n\to H$, where
$$H:=\{g\in\Gamma:v_g\in B\}$$ is a subgroup, and $A_H\subseteq B$. One has $E_{H_n}(a)\to E_H(a)$ in norm, for all $a\in\C[\Gamma]$. Contractivity and density extend this convergence to every $a\in A$. For $b\in B$,
$$
\|b-E_{H_n}(b)\|\leq2d(b,A_{H_n})\longrightarrow0.
$$
Hence $E_H(b)=b$, proving $B=A_H$.\end{proof}

\begin{remark}\label{rem:weak-containment-nonclosed}
The assumption $\lambda_\Gamma\prec\pi$ does not imply that
$\iota_\pi(\Sub(\Gamma))$ is Wijsman-closed. Take
$\Gamma=\mathbb F_\infty=\langle s_1,s_2,\ldots\rangle$,
$\pi=\lambda_\Gamma\oplus1_\Gamma$, and
$H_n=\langle s_n,s_{n+1},\ldots\rangle$.
For any non-amenable group $G$, its trivial character
$\varepsilon_G:C^*(G)\to\mathbb C$ does not factor through $\lambda_G:C^*(G)\to C_r^*(G)$. Thus some $x\in\ker \lambda_G$ satisfies $\varepsilon_G(x)=1$, so
$(0,1)\in(q_G\oplus\varepsilon_G)(C^*(G))$ and
$C^*_{\lambda_G\oplus1_G}(G)=C_r^*(G)\oplus\mathbb C$.
Applying this to $\Gamma$ and $H_n$, and using that
$\lambda_\Gamma|_{H_n}$ is a direct sum of copies of $\lambda_{H_n}$, gives
$$
A=C_r^*(\Gamma)\oplus\mathbb C
\quad\text{and}\quad A_{H_n}=C_r^*(H_n)\oplus\mathbb C.
$$
Put $K_m=\langle s_1,\ldots,s_m\rangle$ and
$B=\mathbb C1\oplus\mathbb C$. For $n>m$, we have
$K_m\cap H_n=\{e\}$, so the canonical conditional expectation $E_{K_m}$
maps $C_r^*(H_n)$ onto $\mathbb C1$. For $a\in C_r^*(K_m)$ and
$b\in C_r^*(H_n)$, $\|a-b\|\geq\|a-E_{K_m}(b)\|\geq d(a,\mathbb C1)$. Since $\mathbb C1\subseteq C_r^*(H_n)$, it follows that, for $z\in\mathbb C$, $d((a,z),A_{H_n})=d(a,C_r^*(H_n))=d(a,\mathbb C1)=d((a,z),B)$. Since the union $\bigcup_m C_r^*(K_m)$ is norm-dense in $C_r^*(\Gamma)$ and distance functions are $1$-Lipschitz, we get $A_{H_n}\to B$. If $A_H=B$, then $\lambda_h$ is scalar for every $h\in H$, hence
$H=\{e\}$. But $A_{\{e\}}=\mathbb C(1,1)\neq B$, so $B$ is not in the image.
\end{remark}

\begin{example}\label{ex:full-group-map-discontinuous}
For $\Gamma=\mathrm{SL}_3(\mathbb Z)$, the full subgroup map $\iota$
is not continuous at $\{I\}$. Write $u_g\in C^*(\Gamma)$ for the canonical unitary associated with $g\in\Gamma$. Define
$N=\ker\bigl(\Gamma\to\mathrm{SL}_3(\mathbb F_3)\bigr)$, $a=I+E_{12}$ and $H_n=\langle a^{3n+1}\rangle$. Note that
$$H_n=\{I+k(3n+1)E_{12}:k\in\mathbb Z\}$$
hence, $H_n\to\{I\}$ in the Chabauty topology.
The finite-index subgroup $N$ has property~$(T)$ since $\Gamma$ has property~$(T)$ (see \cite{BekkaHarpeValette2008}).
Let $p\in C^*(N)\subseteq C^*(\Gamma)$ be its Kazhdan projection
(associated with the trivial representation of $N$, see \cite[Definition~17.2.3, Theorem~17.2.4]{BroOza08}), so $pu_g=p$ for $g\in N$.
Put $x=pu_a$ and $y_n=u_{a^{3n+1}}\in C^*(H_n)$.
Since $a^{3n}\in N$, we have $py_n=x$ and
$$
d(x,C^*(H_n))\leq\|x-\tfrac12y_n\|
=\|(p-\tfrac12 1)y_n\|=\tfrac12.
$$
Let $q:\Gamma\to Q=\Gamma/N$ and let $\rho=\lambda_Q\circ q$
act on $\ell^2(Q)$. Since $\rho|_N$ is trivial, $\rho(p)=1$,
whereas $q(a)\neq e$. Thus, for every $z\in\mathbb C$,
$$
\|x-z1\|\geq\|(\rho(x)-z1)\delta_e\|
=\|\delta_{q(a)}-z\delta_e\|=\sqrt{1+|z|^2}\geq1.
$$
As $\|x\|\leq1$, this gives $d(x,\mathbb C1)=1$, proving discontinuity.
\end{example}

\subsection{The space of von Neumann subalgebras} Let $(M,\tau)$ be a finite von Neumann algebra with faithful normal trace $\tau$ and denote by $\Sub(M)$ the set of unital von Neumann subalgebras of $M$. For $N\in\Sub(M)$, let $E_N:M\to N$ be the unique
$\tau$-preserving conditional expectation. We consider on $\Sub(M)$ the coarsest
topology for which the maps $N\longmapsto\|E_N(x)\|_2$ are continuous, for all $x\in M$. Whenever $M$ has separable predual, the space $\Sub(M)$ is Polish by
\cite[Theorem~6.3]{fima2024michael} and, by
\cite[Theorem~4.3(2)]{FimaLeMaitreMukherjeePatri2026}, this topology is the same as the coarsest
topology for which the maps $N\longmapsto E_N(x)\in{\rm L}^2(M)$ are continuous, for all $x\in M$, where ${\rm L}^2(M)$ has the norm topology. By \cite[Theorem~4.7]{FimaLeMaitreMukherjeePatri2026}, this is the \textbf{Effros--Mar\'echal topology}.

\medskip

\noindent For a Hilbert space $H$, let $\mathcal B(H)$ denote
its algebra of bounded linear operators, and let $C''$ denote the
bicommutant of $C\subseteq\mathcal B(H)$.

\medskip

\noindent Let $A$ be a separable unital $C^*$-algebra with a tracial state
$\tau$. Let $({\rm L}^2(A,\tau),\pi,\xi)$ be the associated GNS representation and set $M=\pi(A)''\subseteq\mathcal{B}({\rm L}^2(A,\tau))$.
The formula $\tau(x)=\langle x\xi,\xi\rangle$ defines a faithful normal
tracial state on $M$.

\begin{proposition}\label{PropBC}
Let $\Phi:\Sub(A)\to\Sub(M)$, $\Phi(B)=\pi(B)''$.
\begin{enumerate}
\item For every $x\in M$, the function
$f_x(B)=\|E_{\Phi(B)}(x)\|_2$ is lower semicontinuous for the lower topology.
\item The map $\Phi$ is Borel for the lower source topology and of Baire
class~$1$ for the Wijsman source topology.
\end{enumerate}
\end{proposition}

\begin{proof}
$(1)$. For $x\in M$, Kaplansky density and orthogonal projection give
$$
f_x(B)^2=\|x\|_2^2-\delta_x(B)^2,
\quad\text{where } \delta_x(B):=\inf_{b\in B}\|x-\pi(b)\|_2.
$$
For $r>0$, the set $U_{x,r}=\{a\in A:\|x-\pi(a)\|_2<r\}$ is norm-open and
$$\{B\in\Sub(A):\delta_x(B)<r\}=\{B\in\Sub(A):B\cap U_{x,r}\neq\emptyset\}.$$
Thus $\delta_x$ is upper semicontinuous and $f_x$ is lower semicontinuous.

\medskip

\noindent$(2)$. For $x\in M$, put $q_x(N)=\|E_N(x)\|_2$.
Since the maps $q_x$ generate the topology of $\Sub(M)$ and $\Sub(M)$ is second countable,
every open subset of $\Sub(M)$ is a countable union of finite intersections of sets $q_x^{-1}((a,b))$, with $a<b$ in $\mathbb R$.
By $(1)$,
$$
\Phi^{-1}\bigl(q_x^{-1}((a,b))\bigr)
=\{f_x>a\}\cap\bigcup_{k\geq1}\{f_x\leq b-1/k\}
$$
is Borel for the lower topology and $F_\sigma$ for the Wijsman topology, since the latter is metrizable.
Countable unions and finite intersections preserve both properties, proving $(2)$.\end{proof}

\noindent Applying Proposition~\ref{PropBC} to the canonical trace on
$C_r^*(\Gamma)$ gives the following.

\begin{corollary}\label{cor:bi-commutant-map-borel}
Let $\Gamma$ be a discrete countable group. The map $$\varphi:\Sub(C^*_r(\Gamma))\rightarrow\Sub({\rm L}(\Gamma)),\quad A\mapsto A''$$
is of Baire class~1 and $\Gamma$-equivariant.
\end{corollary}

\begin{proof} Equivariance follows from
$(\lambda_gB\lambda_g^*)''=\lambda_gB''\lambda_g^*$.\end{proof}

\begin{remark}\label{rem:bicommutant-discontinuous}
Let $X$ be an infinite compact metrizable space and $\mu$ an atomless
Borel probability measure of full support. The map
$B\mapsto B''$ from $\Sub(C(X))$ to $\Sub({\rm L}^\infty(X,\mu))$ is
discontinuous at $\mathbb C1$. By compactness, for each $n\geq1$ choose a
non-empty finite set $F_n\subseteq X$ such that
$X=\bigcup_{y\in F_n}B(y,1/n)$ then, for the Hausdorff distance, one has $d_{\mathcal H}(F_n,X)\leq1/n\to0$. Let $B_n=\{f\in C(X):f|_{F_n}\text{ is constant}\}$.
Proposition~\ref{prop:rad-dist-formula} and
Lemma~\ref{lem:minimal-enclosing-disk} give, for every $f\in C(X)$,
$$
d(f,B_n)=\mathrm{Radius}(f(F_n))
\longrightarrow\mathrm{Radius}(f(X))=d(f,\mathbb C1).
$$
Hence $B_n\to\C1$ in the Wijsman topology. Fix $n$ and put $h_k(x)=\min\{1,k\,d(x,F_n)\}$. Since $\mu(F_n)=0$, for every $f\in C(X)$ and $\xi\in{\rm L}^2(X,\mu)$, dominated convergence gives $\|(h_kf-f)\xi\|_2\to0$. Thus $h_kf\to f$ strongly. Since $h_kf\in B_n$ for all $k,n$, we deduce
$B_n''=C(X)''={\rm L}^\infty(X,\mu)\neq\mathbb C1$ for every $n$.
Consequently, $\Phi$ in Proposition~\ref{PropBC} need not be continuous.
Taking $X=\mathbb T$ with Haar measure and using the Fourier transform gives the same conclusion for $\varphi$ in Corollary~\ref{cor:bi-commutant-map-borel}, with $\Gamma=\mathbb Z$.\end{remark}

\section{Nuclearity}\label{sec: nuc}

\noindent Let $A$ and $B$ be unital $C^*$-algebras. A unital $*$-homomorphism
$\pi:A\to B$ is \textbf{nuclear} if, for every finite $F\subseteq A$ and
$\varepsilon>0$, there are $n\geq1$ and contractive completely positive (CCP)
maps $\phi:A\to M_n(\C)$ and $\psi:M_n(\C)\to B$ such that
$$
\|\psi\phi(a)-\pi(a)\|<\varepsilon\quad\text{for all }a\in F.
$$

\noindent A unital $C^*$-algebra $A$ is called \textbf{nuclear} if the identity map $\id:A\longrightarrow A$ is nuclear. For a discrete group $\Gamma$, amenability of $\Gamma$ is equivalent to nuclearity of $C_r^*(\Gamma)$.

\subsection{Continuous selections}

For $C^*$-algebras $A$ and $B$, let $\mathrm{CCP}(A,B)$ denote the space
of contractive completely positive maps from $A$ to $B$, endowed with the
point-norm topology. If $A$ is finite dimensional, this topology is induced by the operator norm on the space $\mathcal{L}(A,B)$ of  linear maps from $A$ to $B$.

\begin{lemma}\label{lem:selection-dense-unit-ball}
Let $A$ be a separable unital $C^*$-algebra and let $A_1$ be its closed
unit ball. There are Wijsman-continuous maps $g_k:\Sub(A)\to A_1$,
$k\in\N$, such that $\{g_k(B):k\in\N\}$ is dense in $B_1$ for every
$B\in\Sub(A)$.
\end{lemma}

\begin{proof}
By \cite[Corollary~5.5]{fima2024michael} applied to the complete separable
convex set $A_1$, it suffices to show that the map
$\Phi:\Sub(A)\to\mathcal F^*_{\rm conv}(A_1)$, $B\mapsto B_1$ is Wijsman to lower continuous. Let $U\subset A_1$ be open and let us show that
$$\Phi^{-1}(\mathcal{I}_U)=\{B\in\Sub(A):B_1\cap U\neq\emptyset\}$$
is Wijsman open. Fix $C\in \Phi^{-1}(\mathcal{I}_U)$ and let $y\in C_1\cap U$. Since $U$ is open, there exists $\varepsilon>0$ such that $\{x\in A_1:\|x-y\|<\varepsilon\}\subseteq U$. Define the Wijsman-open neighbourhood
$\mathcal{U}:=\{B\in\Sub(A):d(y,B)<\varepsilon/4\}$ of $C$. Let $B\in\mathcal{U}$. There exists $x\in B$ such that $\Vert x-y\Vert<\varepsilon/2$. For $z:=\frac{x}{\max\{1,\|x\|\}}\in B_1$ one has, since $\Vert y\Vert\leq1$,
$$\Vert x-z\Vert=\max\{0,\|x\|-1\}\leq\|x-y\|<\varepsilon/2.$$
It follows that $\Vert z-y\Vert\leq\Vert z-x\Vert+\Vert x-y\Vert<\varepsilon$ hence $z\in U$ so $B\in\Phi^{-1}(\mathcal{I}_U)$.\end{proof}

\begin{lemma}\label{lem:selection-dense-ccp}
Let $A$ be a separable unital $C^*$-algebra and $n\geq1$. There are
Wijsman-continuous maps $f_{n,r}:\Sub(A)\to\mathrm{CCP}(M_n(\C),A)$, for $r\in\N$, such that the set 
$\{f_{n,r}(B):r\in\N\}$ is dense in $\mathrm{CCP}(M_n(\C),B)$ for all $B\in\Sub(A)$.
\end{lemma}

\begin{proof}
For $T\in M_n(A)$, let $Q(T):M_n(\C)\to A$ be the linear map determined by
$$
Q(T)(E_{ij})=
\frac{(T^*T)_{ij}}
{\max\{1,\|\sum_{k=1}^n(T^*T)_{kk}\|\}}
\quad\text{for all }1\leq i,j\leq n.
$$
By Choi's theorem \cite[Proposition~1.5.12]{BroOza08}, $Q(T)$ is completely
positive, and $\|Q(T)\|=\|Q(T)(1)\|\leq1$.
Each $T\mapsto Q(T)(E_{ij})$ is norm-continuous, since its numerator and
denominator are continuous and the latter is at least $1$. The estimate
$$
\|\psi-\phi\|\leq\sum_{i,j=1}^n\|\psi(E_{ij})-\phi(E_{ij})\|
\qquad(\psi,\phi\in\mathcal L(M_n(\C),A))
$$
therefore shows that $Q$ is continuous in the operator norm.

\medskip

\noindent Let $(g_k)$ be as in Lemma~\ref{lem:selection-dense-unit-ball}.
For $\alpha=(k_{ij})_{1\leq i,j\leq n}\in\N^{n^2}$, define
$$
S_\alpha(B)=(g_{k_{ij}}(B))_{1\leq i,j\leq n}\in M_n(B)\quad\text{and}\quad
f_\alpha(B)=Q(S_\alpha(B)).
$$
Each $f_\alpha$ is Wijsman-continuous and takes its value at $B$ in $\mathrm{CCP}(M_n(\C),B)$.

\medskip

\noindent Fix $\phi\in\mathrm{CCP}(M_n(\C),B)$ and set
$S=((\phi(E_{ij}))_{ij})^{1/2}\in M_n(B)_+$. Since
$$
S_{ij}^*S_{ij}\leq\sum_{k=1}^nS_{kj}^*S_{kj}
=\phi(E_{jj})\leq1,
$$
each $S_{ij}$ belongs to $B_1$. By density of the $g_k(B)$ in $B_1$,
there are $\alpha_m\in\N^{n^2}$ such that $S_{\alpha_m}(B)\to S$ in norm.
Moreover, $Q(S)=\phi$, since $S^*S=(\phi(E_{ij}))_{ij}$ and
$\|\phi(1)\|\leq1$. Continuity of $Q$ gives
$f_{\alpha_m}(B)\to\phi$ in operator norm. Enumerating $\N^{n^2}$ yields
the required family.
\end{proof}

\subsection{\texorpdfstring{The set of nuclear $C^*$-subalgebras}{The set of nuclear C*-subalgebras}}\label{subsec: set of nuclear}

It is known that the set of amenable subgroups of a discrete countable group $\Gamma$ is closed in $\Sub(\Gamma)$. More generally, it is shown in \cite[Proposition~4.7]{AHO25} that the set of amenable von Neumann subalgebras of a finite von Neumann algebra $M$ with separable predual is closed in $\Sub(M)$ (see also \cite[Proposition~5.2]{FimaLeMaitreMukherjeePatri2026} for a more general statement and a more elementary proof). In contrast to the finite case, it is shown in \cite[Theorem~5.8]{haagerupEffrosMarechalTopology2000} that every Connes-embeddable $\rm{II}_1$ factor $N\subset\mathcal{B}(H)$, with $H$ separable, is approximable in the Mar\'echal topology on the space of von Neumann subalgebras of $\mathcal{B}(H)$ by amenable von-Neumann algebras.

\medskip

\noindent In this subsection, we show that the set of nuclear $C^*$-subalgebras of a separable unital $C^*$-algebra $A$ is $G_\delta$ in $\Sub(A)$ for the Wijsman topology and not closed in general.

\begin{proposition}\label{thm:nuclear-subalgebras-G-delta}
  Let $A$ be a unital separable $C^*$-algebra.
  Then the set
  $$
    \mathrm{Sub}_{\mathrm{nuc}}(A)\coloneqq \{B\in \mathrm{Sub}(A): B \text{ is nuclear}\}
  $$
  is a $G_{\delta}$ subset of $\mathrm{Sub}(A)$ endowed with the Wijsman topology.
\end{proposition}

\begin{proof}
Choose continuous maps $g_k:\Sub(A)\to A_1$ whose values are dense in
$B_1$ at each $B$, as in Lemma~\ref{lem:selection-dense-unit-ball}.
For each $n$, choose the continuous maps $f_{n,r}$ of
Lemma~\ref{lem:selection-dense-ccp}.

\noindent For $m,q\geq1$, define
$$
U_{m,q}=\bigcup_{\substack{n,r\geq1\\\phi\in\mathrm{CCP}(A,M_n(\C))}}
\left\{B\in\Sub(A):\max_{1\leq k\leq m}
\|f_{n,r}(B)(\phi(g_k(B)))-g_k(B)\|<2^{-q}\right\}.
$$
By continuity of $g_k$ and $f_{n,r}$ and since $\phi$ is a fixed contraction, each set in this union is open. Hence $U_{m,q}$ is open. We claim that $\Sub_{\mathrm{nuc}}(A)=\bigcap_{m,q\geq1}U_{m,q}$.

\medskip

\noindent If $B$ is nuclear, fix $m,q\geq1$ and choose $n\geq1$ and CCP maps
$\phi_0:B\to M_n(\C)$ and $\psi:M_n(\C)\to B$ such that $\|\psi(\phi_0(g_k(B)))-g_k(B)\|<2^{-q}/2$ for all $1\leq k\leq m$. By Arveson's extension theorem, extend $\phi_0$ to a CCP map
$\phi:A\to M_n(\C)$. Choose $r$ so that
$\|f_{n,r}(B)-\psi\|_{\mathcal{L}(M_n(\C),B)}<2^{-q}/2$. For $1\leq k\leq m$,
$$
\|f_{n,r}(B)(\phi(g_k(B)))-g_k(B)\|
\leq\|f_{n,r}(B)-\psi\|+\|\psi(\phi_0(g_k(B)))-g_k(B)\|<2^{-q}.
$$
Hence $B\in U_{m,q}$.

\medskip

\noindent Conversely, suppose that $B\in\bigcap_{m,q}U_{m,q}$. Given a finite
$F\subseteq B_1$ and $\varepsilon>0$, choose for each $x\in F$ an index
$k(x)$ with $\|x-g_{k(x)}(B)\|<\varepsilon/4$. Take
$m\geq\max_{x\in F}k(x)$ and $q$ with $2^{-q}<\varepsilon/2$.
Since $B\in U_{m,q}$ we get $n,r\geq 1$ and $\phi\in\mathrm{CCP}(A,M_n(\C))$ such that, writing $\psi:=f_{n,r}(B)\in\mathrm{CCP}(M_n(\C),B)$ one has $\|\psi\phi(g_k(B))-g_k(B)\|<2^{-q}$ for all $k\leq m$. Hence, for all $x\in F$,
$$
\|\psi\phi(x)-x\|
\leq2\|x-g_{k(x)}(B)\|
   +\|\psi\phi(g_{k(x)}(B))-g_{k(x)}(B)\|<\varepsilon.
$$
Restricting $\phi$ to $B$ proves that $B$ is nuclear.\end{proof}

\begin{example}\label{ex:nuclear-nonclosed-candidate}\label{ex:nuclear-nonclosed}
There is a separable unital $C^*$-algebra $A$ for which $\Sub_{\mathrm{nuc}}(A)$ is not Wijsman-closed. Indeed, let
$B:=C_r^*(\mathbb F_2)\subseteq D:=C(\partial\mathbb F_2)\rtimes_r\mathbb F_2$.
The algebra $D$ is nuclear by \cite[Corollary~2]{Ozawa2007Boundaries}, whereas $B$ is not nuclear. Proposition~\ref{rem:universal-conjugate-approximation} gives $A\supseteq D$ and $B_n=u_nDu_n\to B$ in $\Sub(A)$. Each $B_n\cong D$ is nuclear.
\end{example}

\subsection{\texorpdfstring{Relatively nuclear and exact $C^*$-subalgebras}{Relatively nuclear and exact C*-subalgebras}}
 Let $A$ be a unital $C^{*}$-algebra and let $B$ be a unital $C^{*}$-subalgebra of $A$. We say that $B$ is \textbf{relatively nuclear} in $A$ if the inclusion map $\iota\colon B\hookrightarrow A$ is nuclear. Note that if $A$ is nuclear, then every unital 
$C^*$-subalgebra is relatively nuclear in $A$. 

\begin{proposition}\label{prop:relative-nuclear-closed}
Let $A$ be a unital $C^*$-algebra. The set
$$
\Sub_{\mathrm{rnuc}}(A)=
\{B\in\Sub(A):B\hookrightarrow A\text{ is nuclear}\}
$$
is closed for the lower topology, and hence for the Wijsman topology.
\end{proposition}

\begin{proof}
Let $(B_i)$ be a net in $\Sub_{\mathrm{rnuc}}(A)$ converging to $B$ for
the lower topology. Fix a finite $F\subseteq B$ and $\varepsilon>0$. Since $F$ is finite, the set
$$
\mathcal U=\bigcap_{b\in F}\{C\in\Sub(A):C\cap B(b,\varepsilon/4)\neq\emptyset\}
$$
is a lower-open neighbourhood of $B$.
Lower convergence gives $B_i\in\mathcal U$ for all sufficiently large $i$.
Fix such an $i$ and, for each $b\in F$, choose $b'\in B_i$ with
$\|b-b'\|<\varepsilon/4$. Nuclearity of $B_i\hookrightarrow A$ gives
CCP maps $\phi_0:B_i\to M_n(\C)$ and $\psi:M_n(\C)\to A$ such that
$\|\psi\phi_0(b')-b'\|<\varepsilon/2$ for these finitely many $b'$.
Extend $\phi_0$ to a CCP map $\phi:A\to M_n(\C)$ by Arveson's theorem. Then, for every $b\in F$,
$$
\|\psi\phi(b)-b\|
\leq2\|b-b'\|+\|\psi\phi_0(b')-b'\|<\varepsilon.
$$
Restricting $\phi$ to $B$ proves that the inclusion $B\hookrightarrow A$ is nuclear.\end{proof}

\noindent Recall that a unital $C^*$-algebra $B$ is \textbf{exact} if it admits a
faithful unital representation $\pi:B\to\mathcal B(H)$ which is nuclear. Equivalently, every faithful unital representation is nuclear. In particular, nuclear $C^*$-algebras are exact.

\begin{proposition}\label{prop:exact-subalgebras-closed}
Let $A$ be a unital $C^*$-algebra. The set
$$
\Sub_{\mathrm{ex}}(A)=\{B\in\Sub(A):B\text{ is exact}\}
$$
is closed for the lower topology, and hence for the Wijsman topology.
Moreover, with closure taken in the Wijsman topology,
$\overline{\Sub_{\mathrm{nuc}}(A)}
\subseteq\Sub_{\mathrm{rnuc}}(A)
\subseteq\Sub_{\mathrm{ex}}(A)$.
\end{proposition}

\begin{proof}
View $A\subset\mathcal B(H)$. The inclusion
$j:\Sub(A)\to\Sub(\mathcal B(H))$ is lower-continuous by Proposition~\ref{rem:subspace-closed}. The characterisation of exactness above gives
$$
\Sub_{\mathrm{ex}}(A)
=j^{-1}\bigl(\Sub_{\mathrm{rnuc}}(\mathcal B(H))\bigr).
$$
Proposition~\ref{prop:relative-nuclear-closed}, applied to $\mathcal B(H)$,
therefore proves lower closedness.

\medskip

\noindent We now prove the last statement. For the first inclusion, observe that $$\Sub_{\mathrm{nuc}}(A)\subseteq\Sub_{\mathrm{rnuc}}(A)$$ and use Proposition~\ref{prop:relative-nuclear-closed}. For the second inclusion, note that if $B\hookrightarrow A$ is
nuclear, then its composition with $A\hookrightarrow\mathcal B(H)$ is a faithful
nuclear representation.
\end{proof}

\noindent The preceding results give the following characterisation of exactness.

\begin{corollary}\label{cor:exact-nuclear-wijsman-approximation}
Let $B$ be a separable unital $C^*$-algebra. The following are equivalent:
\begin{enumerate}
\item $B$ is exact.
\item There exist a separable unital $C^*$-algebra $A\supseteq B$ and
nuclear $B_n\in\Sub(A)$ such that $B_n\to B$ in the Wijsman topology.
\item There exist separable unital $C^*$-algebras $B\subseteq D\subseteq A$
with $D\cong\mathcal O_2$ (the Cuntz algebra) and self-adjoint unitaries
$u_n\in A\cap B'$ such that $u_nDu_n\to B$ in the Wijsman topology.
\end{enumerate}
All inclusions are unital; the ambient algebra $A$ may depend on $B$.
\end{corollary}

\begin{proof}
If $B$ is exact, Kirchberg's embedding Theorem provides a unital embedding
$B\hookrightarrow\mathcal O_2$
\cite[Theorem~2.8]{KirchbergPhillips1997Embedding}.
Applying Proposition~\ref{rem:universal-conjugate-approximation} to this
inclusion gives $(1)\Rightarrow(3)$. Since $\mathcal O_2$ is nuclear, $(3)\Rightarrow(2)$. Finally, nuclear algebras are exact, so
Proposition~\ref{prop:exact-subalgebras-closed} gives $(2)\Rightarrow(1)$.
\end{proof}

\section{\texorpdfstring{Random subalgebras and $C^*$-simplicity}{Random subalgebras and C*-simplicity}}\label{sec:random-rigidity}

\subsection{Generalities} For a topological space $X$, let $\Prob(X)$ denote the set of Borel probabilities, endowed with the weak topology, i.e. the coarsest topology for which every map $\nu\mapsto\int_X f\,d\nu$, $f\in C_b(X)$, is continuous. If $X$ is Polish, every $\nu\in\Prob(X)$ is inner regular with respect
to compact sets and outer regular \cite[Theorems~17.10 and~17.11]{Kechris1995}.

\medskip

\noindent For $\nu\in\Prob(X)$, its \textbf{topological support} is the closed set
$$\supp(\nu):=\{x\in X:\nu(U)>0\text{ for every open neighbourhood }U\text{ of }x\}.$$
If $X$ is second countable, in particular if it is Polish, then
$\nu(\supp(\nu))=1$ and $\supp(\nu)$ is the smallest closed subset of $X$
of full $\nu$-measure.

\medskip

\noindent A \textbf{random subgroup (RS)} of a countable
group $\Gamma$ is an element of $\Prob(\Sub(\Gamma))$, where
$\Sub(\Gamma)$ carries the Chabauty topology.

\medskip

\noindent Following \cite{AHO25}, a \textbf{random subalgebra (RA)} of a finite von Neumann algebra $M$ is a Borel probability measure on $\Sub(M)$, with the Effros--Mar\'echal topology. We define a $C^*$-version of this notion.

\begin{definition}
Let $A$ be a unital $C^*$-algebra. A \textbf{random subalgebra (RA)} of $A$ is a Borel probability measure $\nu$ on $\Sub(A)$ endowed with the Wijsman topology.
\end{definition}

\noindent Let $\Gamma$ be a countable group. A unital $C^*$-algebra $A$ equipped with an action
$\alpha:\Gamma\to\Aut(A)$ is called a \textbf{$\Gamma$-$C^*$-algebra}. Note that the action of $\Gamma$ on $\Sub(A)$ defined by $gB:=\alpha_g(B)$, for $g\in\Gamma$, $B\in\Sub(A)$, is by homeomorphisms, by Lemma \ref{lem:inverse-image}.

\medskip

\noindent A finite von Neumann algebra $(M,\tau)$ equipped with a trace preserving action
$\alpha:\Gamma\to\Aut(M)$ is called a \textbf{$\Gamma$-von Neumann algebra}. As in the $C^*$-case, the action of $\Gamma$ on $\Sub(M)$ defined by $gN:=\alpha_g(N)$, for $g\in\Gamma$, $N\in\Sub(M)$, is by homeomorphisms \cite[Section 4]{AmrutamFimaJiang2026}.

\medskip

\noindent For $\nu$ a RA of $D$, where $D$ is either a $\Gamma$-$C^*$-algebra or a $\Gamma$-von Neumann algebra, define $g\nu\in\Prob(\Sub(D))$ by
$g\nu(E):=\nu(\alpha_g^{-1}(E))$ for every Borel set $E\subseteq\Sub(D)$. The action $\alpha:\Gamma\to\operatorname{Aut}(D)$ is called \textbf{inner} if there exists a unitary representation $u:\Gamma\to\mathcal U(D)$ such that $\alpha_g(a)=u_gau_g^*$ for all $g\in\Gamma$ and $a\in D$. Its induced action on $\Sub(D)$ is $gD=u_gDu_g^*$.

\begin{definition}\label{def:random-algebra-invariance}
Let $D$ be either a $\Gamma$-$C^*$-algebra or a $\Gamma$-von Neumann algebra and $\nu\in\Prob(\Sub(D))$.
\begin{itemize}
\item $\nu$ is called \textbf{invariant} ($\IRA$) if $g\nu=\nu$ for
all $g\in\Gamma$.
\item $\nu$ is called \textbf{quasi-invariant} ($\qIRA$) if
$g\nu\sim\nu$ for all $g\in\Gamma$.
\item For $\mu\in\Prob(\Gamma)$,
$\nu$ is called \textbf{$\mu$-stationary} ($\mu\SRA$) if
$$
\nu=\mu*\nu:=\sum_{g\in\Gamma}\mu(g)g\nu.
$$
\item A qIRA $\nu$ is called \textbf{ergodic} if, for every Borel set
$E\subseteq\Sub(A)$,
$$
\bigl[\,\nu(g^{-1}E\triangle E)=0\text{ for all }g\in\Gamma\,\bigr]
\Longrightarrow\nu(E)\in\{0,1\}.
$$
\end{itemize}
\end{definition}

\noindent Note that ergodicity may be tested only on $\Gamma$-invariant Borel sets:
if $E$ satisfies the bracketed condition, then
$E_0:=\bigcap_{g\in\Gamma}g^{-1}E$ is invariant and
$\nu(E\triangle E_0)=0$, since $\Gamma$ is countable.

\medskip

\noindent For a state $\varphi$ on $D$, set $g\varphi=\varphi\circ\alpha_{g^{-1}}$ and
call $\varphi$ \textbf{$\mu$-stationary} if $\varphi=\sum_g\mu(g)g\varphi$.

\medskip

\noindent A probability $\mu\in\Prob(\Gamma)$ is called \textbf{non-degenerate} if $\supp(\mu)$ generates
$\Gamma$ as a semigroup.

\begin{remark}
For any non-degenerate $\mu\in\Prob(\Gamma)$ one has:
$$
\IRA\subseteq\mu\SRA\subseteq\qIRA.
$$
The first inclusion is clear. For the second, if $\nu$ is a $\mu$-SRA then $\nu=\mu^{*n}*\nu$ for every $n$. Fix $g\in\Gamma$. Since $\mu$ is non-degenerate, there exists $n$ with $\mu^{*n}(g)>0$. Then, for any Borel set $E\subseteq\Sub(A)$,  $\nu(E)=\mu^{*n}*\nu(E)=\sum_{h\in\Gamma}\mu^{*n}(h)h\nu(E)\geq\mu^{*n}(g)g\nu(E)$ which implies that $g\nu\ll\nu$. Applying this to $g^{-1}$ and translating gives $\nu\ll g\nu$.\end{remark}

\noindent If $\nu\in\Prob(X)$ and $f:X\longrightarrow Y$ is measurable, its
push-forward is defined by $f_*\nu(E)=\nu(f^{-1}(E))$.

\begin{example}\label{ex-RS}
Let $\pi$ be a unitary representation of $\Gamma$. We will always view $C^*_\pi(\Gamma)$ as an inner $\Gamma$-$C^*$-algebra with the action $gB:=\pi(g)B\pi(g)^*$, for $g\in\Gamma$ and $B\in\Sub(C^*_\pi(\Gamma))$. The map $\iota_\pi:\Sub(\Gamma)\to\Sub(C_\pi^*(\Gamma))$ is of Baire class~1
by Theorem~\ref{thm:reduced-group-map}. Hence every random subgroup
$\nu$ of $\Gamma$ defines the random subalgebras $(\iota_\pi)_*\nu$ of
$C^*_\pi(\Gamma)$. Since $\iota_\pi$ is also $\Gamma$-equivariant,  if $\nu$ is an $\IRS$ then $(\iota_\pi)_*\nu$ is an $\IRA$ and the same statement holds for $\mu$-SRA, $\qIRA$.
\end{example}

\begin{example}
Following \cite{AHO25}, we will always view $\Sub({\rm L}(\Gamma))$ as an inner $\Gamma$-von Neumann algebra with the action $gN:=\lambda_gN\lambda_g^*$, for all $g\in\Gamma$, $N\in\Sub({\rm L}(\Gamma))$. It shown in \cite[Proposition 4.1]{AHO25} (see also \cite[proposition 5.12]{FimaLeMaitreMukherjeePatri2026} for an easier proof) that the map $\Sub(\Gamma)\to\Sub({\rm L}(\Gamma))$ is continuous. It is obviously $\Gamma$-equivariant. Hence, as in Example \ref{ex-RS}, every random subgroup (resp. IRS) defines a random subalgebra (resp. IRA) of ${\rm L}(\Gamma)$.
\end{example}

\noindent The von Neumann assertion in the next Proposition extends the normal-closure construction of \cite[Lemma~4.4]{AHO25}. We provide a simpler proof which unifies both the von Neumann and $C^*$ setting.

\begin{proposition}\label{prop:alg sup unique}
Let $D$ be a separable unital $C^*$-algebra or a finite von Neumann algebra with separable predual and a faithful normal tracial state, and let $\nu\in\Prob(\Sub(D))$.
There is a unique minimal $N\in\Sub(D)$ such that $\nu(\Sub(N))=1$.
If $D$ is a $\Gamma$-$C^*$-algebra or a $\Gamma$-von Neumann algebra and $\nu$ is quasi-invariant, then $N$ is $\Gamma$-invariant.
\end{proposition}

\begin{proof}
The space $\Sub(D)$ is Polish, and $\Sub(C)$ is closed in $\Sub(D)$ for every $C\in\Sub(D)$.
In the $C^*$-case, this follows from Proposition~\ref{rem:subspace-closed}.
In the von Neumann case, if $P_i\subseteq C$ and $P_i\to P$, then
$E_CE_{P_i}=E_{P_i}$; pointwise $2$-norm convergence gives $E_CE_P=E_P$, hence $P\subseteq C$.
Let $N$ be the $C^*$-algebra, respectively von Neumann algebra, generated by
$\bigcup_{B\in\supp(\nu)}B$. Since $\nu(\supp(\nu))=1$, one has $\nu(\Sub(N))=1$.
If $\nu(\Sub(C))=1$, then since $\Sub(C)$ is closed one has $\supp(\nu)\subseteq\Sub(C)$, and $N\subseteq C$.
Finally, quasi-invariance gives $g\supp(\nu)=\supp(\nu)$ and therefore $gN=N$.
\end{proof}

\begin{definition}\label{def:alg supp}
For $D$ and $\nu$ as in Proposition~\ref{prop:alg sup unique}, the unique minimal $N\in\Sub(D)$ such that $\nu(\Sub(N))=1$ is the \textbf{algebraic support} of $\nu$, denoted $\langle\nu\rangle$.
The measure $\nu$ is \textbf{spanning} if $\langle\nu\rangle=D$.
\end{definition}

\subsection{Nuclear and amenable random subalgebras}

Let $A$ be a unital separable $C^*$-algebra and recall that the set $\Sub_{\mathrm{nuc}}(A)$ is Borel by
Proposition~\ref{thm:nuclear-subalgebras-G-delta}. A random subalgebra $\nu$ is called
\textbf{nuclear} if $\nu(\Sub_{\mathrm{nuc}}(A))=1$.

\begin{remark}\label{RmkNucRA}
The following holds.
\begin{enumerate}
\item Every nuclear RA $\nu$ of $A$ satisfies
$\supp(\nu)\subseteq\Sub_{\mathrm{rnuc}}(A)\subseteq\Sub_{\mathrm{ex}}(A)$.
Indeed, $\Sub_{\mathrm{rnuc}}(A)$ is closed by
Proposition~\ref{prop:relative-nuclear-closed}, and
$$\Sub_{\mathrm{nuc}}(A)\subseteq \Sub_{\mathrm{rnuc}}(A)\subseteq\Sub_{\mathrm{ex}}(A)$$
by
Proposition~\ref{prop:exact-subalgebras-closed}.
Thus $\nu(\Sub_{\mathrm{rnuc}}(A))=1$ so $\supp(\nu)\subseteq \Sub_{\mathrm{rnuc}}(A)$. However, $\supp(\nu)$ need not consist of nuclear algebras: for $B_n\to B$ as in Example~\ref{ex:nuclear-nonclosed}, the measure $\sum_{n\geq1}2^{-n}\delta_{B_n}$ is nuclear and $B\in\supp(\nu)$ is not nuclear.

\item\label{item:spanning-nuclear-qira} For every action $\alpha:\Gamma\to\operatorname{Aut}(A)$ of a countable group on a separable unital $C^*$-algebra, there exists a spanning qIRA concentrated on abelian subalgebras (hence a nuclear spanning qIRA).
\end{enumerate}
Indeed, choose self-adjoint generators $(a_n)_{n\geq1}$ of $A$ and $m\in\Prob(\Gamma)$ with $m(g)>0$ for all $g\in\Gamma$, and set $\nu:=\sum_{n\geq1}\sum_{g\in\Gamma}2^{-n}m(g)\delta_{\alpha_g(C^*(1,a_n))}$. The set
$$
\mathcal D=\{\alpha_g(C^*(1,a_n)):n\geq1,\ g\in\Gamma\}
$$
is countable and $\Gamma$-invariant, and every point of $\mathcal D$ has positive $\nu$-mass. Hence, for every Borel set $E\subseteq\Sub(A)$ and every $h\in\Gamma$,
$$
\nu(E)=0\ \Leftrightarrow\ E\cap\mathcal D=\emptyset
\ \Leftrightarrow\ h\nu(E)=0.
$$
Thus $h\nu\sim\nu$. Each atom is abelian, so $\nu$ is nuclear. If $C\in\Sub(A)$ satisfies $\nu(\Sub(C))=1$, then $C^*(1,a_n)\subseteq C$ for every $n$, since $m(e)>0$. Consequently $C=A$, and $\langle\nu\rangle=A$.

\begin{enumerate}\setcounter{enumi}{2}
\item The existence assertion in~\eqref{item:spanning-nuclear-qira} fails with qIRA replaced by $\mu$-SRA, even for symmetric full-support $\mu$. Remark~\ref{rem:spanning-stationary-obstruction} gives a class of pairs $(A,\mu)$ for which no spanning nuclear $\mu$-SRA exists.
\end{enumerate}
\end{remark}

\noindent A von Neumann algebra $N$ is called \textbf{amenable} if every embedding $N\subseteq\mathcal B(H)$ admits a conditional expectation $E:\mathcal B(H)\to N$. For a discrete group $\Gamma$, amenability of $\Gamma$ is equivalent to amenability of ${\rm L}(\Gamma)$. We shall use the following well-known characterizations of nuclearity.

\begin{theorem}[{\cite[Theorems~IV.3.1.5 and~IV.2.2.13]{Blackadar2006}}]\label{ThmNuc}
Let $A$ be a unital $C^*$-algebra. The following are equivalent.
\begin{enumerate}
    \item $A$ is nuclear.
    \item $A^{**}$ is amenable.
    \item For any unital $*$-homomorphism $\pi:A\to\mathcal B(H)$, $\pi(A)''$ is amenable.
\end{enumerate}
\end{theorem}

\medskip

\noindent Let $(M,\tau)$ be a finite von Neumann algebra with separable predual and let
$\Sub_{\mathrm{am}}(M)$ denote its amenable unital von Neumann subalgebras.
This set is closed in $\Sub(M)$ by \cite[Proposition~4.7]{AHO25}
(see also \cite[Proposition~5.2]{FimaLeMaitreMukherjeePatri2026} for a more general statement and an easier proof).

\medskip

\noindent In contrast with Remark \ref{RmkNucRA}, we obtain the following.

\begin{proposition}\label{prop:amenable-tracial-closures}
Let $A$ be a separable unital $C^*$-algebra with a tracial state $\tau$,
and let $\pi_\tau$ be its GNS representation. Then
$$
\Sub_{{\rm am}}^\tau(A)
:=\{B\in\Sub(A):\pi_\tau(B)''\text{ is amenable}\}
$$
is closed for the lower topology, hence for the Wijsman topology.
Consequently, for every $\nu\in\Prob(\Sub(A))$, the following are equivalent.
\begin{enumerate}
\item $\nu(\Sub_{{\rm am}}^\tau(A))=1$.
\item $\supp(\nu)\subseteq\Sub_{{\rm am}}^\tau(A)$.
\end{enumerate}
In particular, every nuclear random subalgebra $\nu$ satisfies these two conditions.
\end{proposition}

\begin{proof}
Let $(B_i)$ be a net in $\Sub_{{\rm am}}^\tau(A)$ converging to $B$
for the lower topology. Put $M=\pi_\tau(A)''$, $N_i=\pi_\tau(B_i)''$
and $N=\pi_\tau(B)''$. Equip $M$ with the faithful normal trace induced
by $\tau$. For each unital von Neumann subalgebra $P\subseteq M$, let
$E_P:M\to P$ be the trace-preserving conditional expectation. For every $b\in B$,
$$
\|\pi_\tau(b)-E_{N_i}(\pi_\tau(b))\|_{2,\tau}
=d_{\Vert\cdot\Vert_{2,\tau}}(\pi_\tau(b),\mathrm{L}^2(N_i))\leq d_A(b,B_i)\longrightarrow0.
$$
Kaplansky density and $2$-norm contractivity give
$E_{N_i}(x)\to x$ in $2$-norm for every $x\in N$. By amenability of $N_i$, each $T_i=E_N\circ E_{N_i}|_N$ is a point-ultraweak limit of normal UCP maps factoring through matrix algebras. Since $T_i\to\id_N$ point-ultraweakly, the same holds for $\id_N$. Thus $N$ is amenable, by the argument of \cite[Definition~5.1 and proof of Proposition~5.2]{FimaLeMaitreMukherjeePatri2026}. This proves closedness. Let us now show that $(1)\Leftrightarrow(2)$. If $\nu(Sub_{{\rm am}}^\tau(A))=1$ then, since $\Sub_{{\rm am}}^\tau(A)$ is closed we have $\supp(\nu)\subseteq \Sub_{{\rm am}}^\tau(A)$. Conversely, if $\supp(\nu)\subseteq\Sub_{{\rm am}}^\tau(A)$, then
$1=\nu(\supp(\nu))\leq\nu(\Sub_{{\rm am}}^\tau(A))\leq1$ so $\nu(\Sub_{{\rm am}}^\tau(A))=1$. Finally, $\Sub_{\mathrm{nuc}}(A)\subseteq\Sub_{{\rm am}}^\tau(A)$
by Theorem~\ref{ThmNuc}.
\end{proof}

\noindent Let $\Gamma$ be countable discrete group. The canonical trace on its von Neumann algebra ${\rm L}(\Gamma)$ and its reduced $C^*$-algebra $C^*_r(\Gamma)$ is denoted $\tau_0$ in both cases.

\medskip

\noindent A random subalgebra $\nu$ of ${\rm L}(\Gamma)$ is called \textbf{amenable} if $\nu(\Sub_{\mathrm{am}}({\rm L}(\Gamma)))=1$. In this case we may view $\nu\in\Prob(\Sub_{\mathrm{am}}({\rm L}(\Gamma)))$. 

\medskip

\noindent For $N\in\Sub(\mathrm L(\Gamma))$, the function
$\phi_N(g)=\|E_N(\lambda_g)\|_2^2$ and the continuous equivariant map
$N\mapsto\phi_N$ into normalized positive-definite functions are considered in
\cite[Definition~4.1, Remark~4.2 and Proposition~4.3]{AmrutamJiang2026Confined}.
The next Lemma proves that integration of $N\mapsto\phi_N$ against an amenable random subalgebra of ${\rm L}(\Gamma)$ gives a state of $C_r^*(\Gamma)$.

\begin{lemma}\label{lem:stationary-state-from-nuclear-sra}
For all $\nu\in\Prob(\Sub_{\mathrm{am}}({\rm L}(\Gamma)))$ there exists a unique state
$$\psi_\nu\in S(C_r^*(\Gamma))\quad\text{such that}\quad\psi_\nu(\lambda_g)=\int\|E_N(\lambda_g)\|_2^2\,d\nu(N)\quad\text{ for all }g\in\Gamma.$$
Moreover, the map $\Prob(\Sub_{\mathrm{am}}({\rm L}(\Gamma))))\to S(C_r^*(\Gamma))$, $\nu\longmapsto\psi_\nu$ is affine, weak to weak* continuous, and $\Gamma$-equivariant. Viewing $\nu\in\Prob(\Sub({\rm L}(\Gamma))$, the following holds.
\begin{enumerate}
\item If $\nu$ is invariant then $\psi_\nu$ is a trace.
\item If $\nu$ is $\mu$-stationary then $\psi_\nu$ is a $\mu$-stationary state.
\end{enumerate}
\end{lemma}

\begin{proof}
The uniqueness of $\psi_\nu$ is clear. Let us show the existence. During the proof we write $M={\rm L}(\Gamma)$ and fix $N\in\Sub_{\mathrm{am}}(M)$. On
$\mathcal H_N=\mathrm{L}^2(M)\otimes_N \mathrm{L}^2(M)$, put $\xi_N=1\otimes_N1$ and $\sigma_N(g)(x\otimes_N y)=\lambda_gx\otimes_N y\lambda_g^*$. One has $\|\xi_N\|=1$ and $\langle\sigma_N(g)\xi_N,\xi_N\rangle
=\tau_0(E_N(\lambda_g)\lambda_g^*)=\|E_N(\lambda_g)\|_2^2.$ By \cite[Theorem~13.4.1, p.~231]{ADP} and induction of correspondences
\cite[Exercise~13.15(ii), p.~239]{ADP},
$$
{}_N\mathrm{L}^2(N)_N\prec{}_N(\mathrm{L}^2(N)\otimes \mathrm{L}^2(N))_N
\quad\text{and}\quad
{}_M\mathcal H_N{}_M\prec{}_M(\mathrm{L}^2(M)\otimes \mathrm{L}^2(M))_M.
$$
Identify $\mathrm L^2(M)$ with $\ell^2(\Gamma)$, and use the convention $\rho_g\delta_t=\delta_{tg^{-1}}$.
The representation of $\Gamma$ on $\mathrm{L}^2(M)\otimes\mathrm{L}^2(M)$
given by $g\mapsto(\xi\otimes\eta\mapsto\lambda_g\xi\otimes\eta\lambda_g^*)$
is $\lambda\otimes\rho$. By Fell's Lemma, $\lambda\ot\rho$ is unitarily equivalent to $\lambda\ot 1$. Weak containment of correspondences therefore gives, for every finitely supported family $(c_g)_{g\in\Gamma}$,
$$
\left\|\sum_g c_g\sigma_N(g)\right\|
\leq\left\|\sum_g c_g(\lambda_g\otimes\rho_g)\right\|
=\left\|\sum_g c_g\lambda_g\ot 1\right\|=\left\|\sum_g c_g\lambda_g\right\|.
$$
Hence $\sigma_N\prec\lambda_\Gamma$, and its coefficient defines a state $\psi_N$ on $C_r^*(\Gamma)$. Each map $N\mapsto\psi_N(\lambda_g)$ is continuous by definition of the Effros--Mar\'echal topology, and bounded.
By norm density of $\mathbb C[\Gamma]$ and $\|\psi_N\|=1$,  $N\mapsto\psi_N(x)$, $x\in C_r^*(\Gamma)$ is continuous for all $x\in C_r^*(\Gamma)$. Hence we can define $\psi_\nu(x):=\int\psi_N(x)d\nu(N)$ for all $x\in C^*_r(\Gamma)$ and it is easy to check that $\psi_\nu$ is a state on $C^*_r(\Gamma)$ and that the map $\nu\mapsto\psi_\nu$ is weak to weak* continuous and affine. For $s\in\Gamma$, uniqueness of the trace-preserving conditional expectation gives $E_{sN}=\Ad(\lambda_s)\circ E_N\circ\Ad(\lambda_{s^{-1}})$. Since $\Ad(\lambda_s)$ preserves the $2$-norm we have $
\psi_{sN}(\lambda_g)=\psi_N(\lambda_{s^{-1}gs})$, for $s,g\in\Gamma$, and by integrating this proves the $\Gamma$-equivariance: $\psi_{s\nu}=\psi_\nu\circ\Ad(\lambda_{s^{-1}})$.
Equivariance gives the stationarity assertion $(2)$. For invariant $\nu$ in $(1)$, conjugation invariance of $\psi_\nu$ implies $\psi_\nu(\lambda_g\lambda_h)=\psi_\nu(\lambda_h\lambda_g)$; linearity and norm density imply that $\psi_\nu$ is tracial.\end{proof}

\begin{remark}\label{rem:cstar-state-map}
For every RA $\nu$ of $C_r^*(\Gamma)$ satisfying
$\nu(\Sub_{{\rm am}}^{\tau_0}(C_r^*(\Gamma)))=1$, define
$\psi_\nu:=\psi_{\varphi_*\nu}$, where $\varphi(B)=B''$.
By Corollary~\ref{cor:bi-commutant-map-borel} and
Lemma~\ref{lem:stationary-state-from-nuclear-sra}, this map is affine and
$\Gamma$-equivariant; it sends invariant laws to traces and
$\mu$-stationary laws to $\mu$-stationary states.
By Theorem~\ref{ThmNuc}, it is defined for every nuclear RA.

\medskip

\noindent This map need not be continuous for the weak topology on $\Prob(\Sub(C^*_r(\Gamma)))$, where $\Sub(C^*_r(\Gamma))$ has the Wijsman topology. Indeed, take $\Gamma=\Z$ and identify
$C_r^*(\Z)=C(\mathbb T)$ and $\lambda_1(z)=z$.
By Remark~\ref{rem:bicommutant-discontinuous}, there are nuclear
$B_n\in\Sub(C(\mathbb T))$ such that $B_n\to\C1$ and
$B_n''={\rm L}^\infty(\mathbb T)$, with respect to Haar measure. Hence
$$
\delta_{B_n}\longrightarrow\delta_{\C1},\qquad
\psi_{\delta_{B_n}}(\lambda_1)=1,\qquad
\psi_{\delta_{\C1}}(\lambda_1)=0.
$$
\end{remark}

\subsection{Concentration on the amenable radical}
The following is a $C^*$-version of \cite[Theorem~B]{AHO25}. We give a proof using the state construction of Lemma~\ref{lem:stationary-state-from-nuclear-sra}. Recall that the \textbf{amenable radical} $\Rad(\Gamma)$ is the largest amenable normal subgroup of $\Gamma$.

\begin{theorem}\label{thm: nuclear IRA}
Let $\nu$ be an IRA of $C_r^*(\Gamma)$. If
$\nu(\Sub_{{\rm am}}^{\tau_0}(C_r^*(\Gamma)))=1$ then
$$
\langle\nu\rangle\subseteq C_r^*(\Rad(\Gamma)).
$$
\end{theorem}

\begin{proof}
The map $\varphi:B\mapsto B''$ is Borel and $\Gamma$-equivariant by
Corollary~\ref{cor:bi-commutant-map-borel}, so $\varphi_*\nu$ is an amenable IRA
on $\mathrm L(\Gamma)$. Lemma~\ref{lem:stationary-state-from-nuclear-sra}
therefore gives a tracial state $\psi_\nu:=\psi_{\varphi_*\nu}$ on $C_r^*(\Gamma)$.
By \cite[Theorem~4.1]{breuillard2017c}, every tracial state on $C_r^*(\Gamma)$
vanishes on $\lambda_g$ for $g\notin\Rad(\Gamma)$. Hence, for every such $g$,
$$
0=\psi_\nu(\lambda_g)
=\int_{\Sub(C_r^*(\Gamma))}\|E_{B''}(\lambda_g)\|_2^2\,d\nu(B).
$$
Hence, for all $g\notin\Rad(\Gamma)$, one has $E_{B''}(\lambda_g)=0$, for $\nu$-almost every $B$. Since $\Gamma$ is countable we deduce that, for $\nu$-almost every $B$,
$E_{B''}(\lambda_g)=0$ for all $g\notin\Rad(\Gamma)$. For every such $B$ and every $b\in B$ one has:
$$
\langle b,\lambda_g\rangle_2=\tau_0(\lambda_g^*b)=\tau_0(E_{B''}(\lambda_g)^*b)=0\quad\text{for all}\quad g\notin\Rad(\Gamma).$$
Thus the Fourier coefficients of $b$ vanish outside $\Rad(\Gamma)$,
so $E_{\mathrm L(\Rad(\Gamma))}(b)=b$.
Since this expectation maps $C_r^*(\Gamma)$ onto $C_r^*(\Rad(\Gamma))$,
we obtain $B\subseteq C_r^*(\Rad(\Gamma))$ for $\nu$-almost every $B$.
Proposition~\ref{prop:alg sup unique} then gives
$\langle\nu\rangle\subseteq C_r^*(\Rad(\Gamma))$.
\end{proof}

\begin{remark}
Note that Theorem \ref{thm: nuclear IRA} applies to any nuclear IRA of $C^*_r(\Gamma)$. Moreover, the proof of Theorem \ref{thm: nuclear IRA} could also be deduced from \cite[Theorem~B]{AHO25} by using the map $B\mapsto B''$. Conversely, the argument using Lemma~\ref{lem:stationary-state-from-nuclear-sra}
and \cite[Theorem~4.1]{breuillard2017c}, applied directly to an amenable IRA on
$\mathrm L(\Gamma)$, gives an alternative proof of \cite[Theorem~B]{AHO25}.
\end{remark}

\begin{corollary}\label{thm:maximal-nuclear-invariant}
The algebra $C_r^*(\Rad(\Gamma))$ is the maximal nuclear $\Gamma$-invariant unital
$C^*$-subalgebra of $C_r^*(\Gamma)$.
\end{corollary}

\begin{proof}
The algebra $C_r^*(\Rad(\Gamma))$ is nuclear and invariant. For every nuclear
invariant $B$, apply Theorem~\ref{thm: nuclear IRA} to $\delta_B$.
\end{proof}

\begin{remark}
Corollary \ref{thm:maximal-nuclear-invariant} could also be deduced from \cite[Theorem~A]{AHO25}. Indeed, if $B\subset C^*_r(\Gamma)$ is nuclear and $\Gamma$-invariant, then $B''\subseteq{\rm L}(\Gamma)$ is amenable by Theorem~\ref{ThmNuc} and
$\Gamma$-invariant, so \cite[Theorem~A]{AHO25} gives
$B''\subseteq\mathrm L(\Rad(\Gamma))$. Hence, $B\subseteq C_r^*(\Gamma)\cap\mathrm L(\Rad(\Gamma))=C_r^*(\Rad(\Gamma))$.
\end{remark}

\begin{corollary}\label{cor:nuclear-ira-regular}
The amenable radical of $\Gamma$ is trivial if and only if $\delta_{\C1}$
is the only nuclear $\IRA$ of $C_r^*(\Gamma)$.
\end{corollary}

\begin{proof}
If $\Rad(\Gamma)=\{e\}$, Theorem~\ref{thm: nuclear IRA} gives
$\langle\nu\rangle\subseteq\C1$ for every nuclear IRA $\nu$, so
$\nu=\delta_{\C1}$. Conversely, if $\Rad(\Gamma)\neq\{e\}$, then
$\delta_{C_r^*(\Rad(\Gamma))}$ is a non-trivial nuclear IRA.
\end{proof}

\begin{remark}\label{rem:arbitrary-representation-counterexample}
The preceding corollary fails for $C^*_\pi(\Gamma)$, where $\pi$ is an arbitrary unitary representation.
Compose a surjection $\mathbb F_2\to S_3$ with the irreducible
two-dimensional representation of $S_3$. Then
$C_\pi^*(\mathbb F_2):=C^*(\pi(\mathbb F_2))=M_2(\mathbb C)$,
so $\delta_{M_2(\mathbb C)}$
is a non-trivial nuclear IRA, although $\mathbb F_2$ is $C^*$-simple
\cite{Powers75} (hence has trivial amenable radical).
\end{remark}

\subsection{\texorpdfstring{Stationary rigidity and $C^*$-simplicity}{Stationary rigidity and C*-simplicity}}

\noindent We will use the following result, due to Hartman and Kalantar, who proved the equivalence of $(1)$ and $(2)$ below
in \cite[Theorem~5.2]{hartman2023stationary}. The strengthened implication
$(1)\Rightarrow(3)$ follows from
\cite[Remark~1 and Lemma~8]{alpeev2025characterization}.

\medskip

\noindent A probability measure $\mu\in\Prob(\Gamma)$ is called \textbf{symmetric}
if $\mu(g)=\mu(g^{-1})$ for every $g\in\Gamma$. It has \textbf{full support} if $\mu(g)>0$ for every $g\in\Gamma$.

\begin{theorem}[{Hartman--Kalantar; Alpeev}]
\label{thm:HK-nondegenerate}
Let $\Gamma$ be a countable discrete group. The following are equivalent.
\begin{enumerate}
\item $\Gamma$ is $C^*$-simple.
\item There exists $\mu\in\Prob(\Gamma)$ such that the canonical trace $\tau_0$
is the unique $\mu$-stationary state on $C_r^*(\Gamma)$.
\item There exists a symmetric probability measure
$\mu\in\Prob(\Gamma)$ with full support such that $\tau_0$
is the unique $\mu$-stationary state on $C_r^*(\Gamma)$.
\end{enumerate}
\end{theorem}

\noindent The main result of this subsection is the following characterisation of C*-simplicity.

\begin{theorem}\label{thm:simple}
For a countable group $\Gamma$, the following are equivalent.
\begin{enumerate}
\item $\Gamma$ is $C^*$-simple.
\item For some $\mu\in\Prob(\Gamma)$, $\delta_{\mathbb C1}$ is the unique nuclear $\mu$-SRA of $C_r^*(\Gamma)$.
\item For some symmetric full-support $\mu\in\Prob(\Gamma)$, $\delta_{\mathbb C1}$ is the unique nuclear $\mu$-SRA of $C_r^*(\Gamma)$.
\item For some $\mu\in\Prob(\Gamma)$, $\delta_{\mathbb C1}$ is the unique amenable $\mu$-SRA of $\mathrm L(\Gamma)$.
\item For some symmetric full-support $\mu\in\Prob(\Gamma)$, $\delta_{\mathbb C1}$ is the unique amenable $\mu$-SRA of $\mathrm L(\Gamma)$.
\end{enumerate}
In $(2)$ and $(3)$, nuclearity may be replaced by the condition
$\nu(\Sub_{{\rm am}}^{\tau_0}(C_r^*(\Gamma)))=1$.
Moreover, the uniqueness conclusions in $(2)$--$(5)$, including
this replacement, hold {for every $\mu\in\Prob(\Gamma)$} for which $\tau_0$
is the unique $\mu$-stationary state on $C_r^*(\Gamma)$.
\end{theorem}

\begin{proof}
Let $\mu\in\Prob(\Gamma)$ be such that $\tau_0$ is the unique $\mu$-stationary state on $C_r^*(\Gamma)$. For every amenable $\mu$-SRA $\nu$ of $\mathrm L(\Gamma)$, Lemma~\ref{lem:stationary-state-from-nuclear-sra} gives $\psi_\nu=\tau_0$, hence
$$
\int\|E_N(\lambda_g)\|_2^2\,d\nu(N)=0\quad\text{for all}\quad g\neq e.
$$
Since $\Gamma$ is countable one has, for $\nu$-almost every $N$, $E_N(\lambda_g)=0$ for all $g\neq e$. Fix such an $N$ and let $x\in N$. Then, $\tau_0(\lambda_g^*x)=\tau_0(E_N(\lambda_g)^*x)=0$ for all $g\neq e$. Hence, $x=\tau_0(x)1$ so $N=\C1$ and $\nu=\delta_{\mathbb C1}$.
If $\eta$ is a $\mu$-SRA of $C_r^*(\Gamma)$ with
$\eta(\Sub_{\mathrm{am}}^{\tau_0}(C_r^*(\Gamma)))=1$, the equivariant Borel map $\varphi:B\mapsto B''$ sends $\eta$ to such an amenable $\mu$-SRA of ${\rm L}(\Gamma)$. Since $B''=\mathbb C1$ implies $B=\mathbb C1$, we deduce from the discussion above that $\eta=\delta_{\mathbb C1}$. This proves the final assertion.

\medskip
\noindent Theorem~\ref{thm:HK-nondegenerate} now gives $(1)\Rightarrow(5)$.
The preceding pushforward argument gives $(5)\Rightarrow(3)$ and $(4)\Rightarrow(2)$, including the stated replacement of nuclearity.
The implications $(5)\Rightarrow(4)$ and $(3)\Rightarrow(2)$ are immediate.

\medskip
\noindent For $(2)\Rightarrow(1)$, suppose that $\Gamma$ is not $C^*$-simple and fix any $\mu\in\Prob(\Gamma)$. By \cite[Corollary~5.7]{hartman2023stationary}, there is an amenable $\mu$-stationary random subgroup $\eta\neq\delta_{\{e\}}$.
The equivariant Borel injection $H\mapsto C_r^*(H)$ of Theorem~\ref{thm:reduced-group-map} sends $\eta$ to a nuclear $\mu$-SRA distinct from $\delta_{\mathbb C1}$. Thus $(2)$ fails, also with the stated replacement.
\end{proof}

\begin{remark}\label{rem:simple-representation-counterexample}
The implication $(1)\Rightarrow(2)$ in Theorem~\ref{thm:simple} fails
if $C_r^*(\Gamma)$ is replaced by $C_\pi^*(\Gamma)$ and $(1)$ is replaced
by simplicity of $C_\pi^*(\Gamma)$, even when $C^*_\pi(\Gamma)$ is infinite dimensional.
Fix $\theta\in\mathbb R\setminus\mathbb Q$ and define unitaries on
$\ell^2(\mathbb Z)$ by
$$
U\delta_n=\delta_{n+1}\quad\text{and}\quad
V\delta_n=e^{2\pi i n\theta}\delta_n\quad\text{for all}\quad n\in\mathbb Z.
$$
Let $\Gamma=\mathbb F_2=\langle a,b\rangle$ and let $\pi$ be the unitary
representation determined by $\pi(a)=U$, $\pi(b)=V$. Since
$VU=e^{2\pi i\theta}UV$, the algebra
$A:=C_\pi^*(\Gamma)=C^*(U,V)$ is an irrational rotation algebra and is
simple. The bilateral shift $U$ has spectrum $\mathbb T$, so
$B:=C^*(U)\simeq C(\mathbb T)$ is nuclear and infinite-dimensional.
Moreover,
$$
\mathbb C1\subsetneq B\subsetneq A,\qquad
UBU^*=B,\qquad VBV^*=B.
$$
Thus $A$ is simple and
infinite-dimensional, but $\delta_B\neq\delta_{\mathbb C1}$ is an
ergodic nuclear IRA of $A$, hence a nuclear $\mu$-SRA for every
$\mu\in\Prob(\Gamma)$. In particular, the modified assertions $(2)$
and $(3)$ fail even when restricted to ergodic measures.
\end{remark}

\begin{remark}\label{rem:spanning-stationary-obstruction}
Let $\Gamma\neq\{e\}$ be countable and $C^*$-simple, and let $\mu\in\Prob(\Gamma)$ be such that $\tau_0$ is the unique $\mu$-stationary state on $C_r^*(\Gamma)$. Let $A$ be a separable unital $\Gamma$-$C^*$-algebra and $q:A\to C_r^*(\Gamma)$ an equivariant unital surjective $*$-homomorphism, with conjugation on the target. Every nuclear $\mu$-SRA $\nu$ of $A$ satisfies
$$
\langle\nu\rangle\subseteq q^{-1}(\mathbb C1)\subsetneq A;
$$
hence $A$ has no spanning nuclear $\mu$-SRA.

\medskip
\noindent By Lemma~\ref{lem:direct-image}, $q_{\mathrm{im}}:B\mapsto q(B)$ is Borel. It is equivariant and preserves nuclearity, since
$q(B)\simeq B/(B\cap\ker q)$. Thus $(q_{\mathrm{im}})_*\nu$ is a nuclear $\mu$-SRA of $C_r^*(\Gamma)$, and Theorem~\ref{thm:simple} gives $q(B)=\mathbb C1$ for $\nu$-almost every $B$. Proposition~\ref{prop:alg sup unique} yields the inclusion above.

\medskip
\noindent This applies to $C_r^*(\Gamma)$ and $C^*(\Gamma)$ with their conjugation actions, in particular for $\Gamma=\mathbb F_2$ \cite{Powers75}. By Corollary~\ref{rem:generic-nuclear-stationarity}, the non-existence conclusion holds for a dense $G_\delta$ set of full-support measures in $\Prob(\Gamma)$, and likewise in its symmetric subspace.

\medskip
\noindent For $\mu=\delta_e$, every RA is stationary, so Remark~\ref{RmkNucRA}\eqref{item:spanning-nuclear-qira} gives a spanning nuclear $\mu$-SRA for every separable unital $\Gamma$-$C^*$-algebra.
\end{remark}

\begin{corollary}\label{rem:generic-nuclear-stationarity}
Let $\Gamma$ be a countable $C^*$-simple group, and let $\mathcal P$ be either
$\Prob(\Gamma)$ or its subspace of symmetric probability measures, with the
topology induced by the $\ell^1$-norm.
There exists a dense $G_\delta$ subset $\mathcal G\subseteq\mathcal P$
consisting of full-support measures such that, for every $\mu\in\mathcal G$:
\begin{enumerate}
\item Every amenable $\mu$-SRA of $\mathrm L(\Gamma)$ equals $\delta_{\mathbb C1}$.
\item Every $\mu$-SRA $\nu$ of $C_r^*(\Gamma)$ satisfying
$\nu(\Sub_{{\rm am}}^{\tau_0}(C_r^*(\Gamma)))=1$ equals $\delta_{\mathbb C1}$.
\end{enumerate}
\end{corollary}

\begin{proof}
By \cite[Corollary~6 and Lemma~8]{alpeev2025characterization}, there is a
dense $G_\delta$ subset $\mathcal H\subseteq\mathcal P$ such that $\tau_0$
is the unique $\mu$-stationary state on $C_r^*(\Gamma)$ for every $\mu\in\mathcal H$.
For each $g\in\Gamma$, the set $U_g=\{\mu\in\mathcal P:\mu(g)>0\}$ is open
and dense: if $\rho$ is a fixed symmetric full-support probability, then
$(1-\varepsilon)\mu+\varepsilon\rho\in U_g$ and converges to $\mu$ in
$\ell^1$ as $\varepsilon\downarrow0$.
Since $\mathcal P$ is closed in $\ell^1(\Gamma)$, the Baire category theorem therefore gives a dense $G_\delta$ set
$$
\mathcal G=\mathcal H\cap\bigcap_{g\in\Gamma}U_g.
$$
Every $\mu\in\mathcal G$ has full support, and the final assertion of
Theorem~\ref{thm:simple} gives $(1)$ and $(2)$.
\end{proof}

\noindent The next proposition applies whenever every closed two-sided ideal is $\Gamma$-invariant, in particular for inner actions.

\begin{proposition}\label{prop: simple implies trivial IRAs}
Let $A\neq\mathbb C$ be a separable unital $\Gamma$-$C^*$-algebra whose closed two-sided ideals are all $\Gamma$-invariant.
\begin{enumerate}
\item If $A$ has no characters and its only ergodic IRAs are
$\delta_{\mathbb C1}$ and $\delta_A$, then $A$ is simple.
\item If $\delta_{\mathbb C1}$ is the only nuclear IRA of $A$, then
$A$ has no non-zero nuclear closed two-sided ideals.
\end{enumerate}
\end{proposition}

\begin{proof}
By assumption, every closed two-sided ideal $I\subseteq A$ is $\Gamma$-invariant. Thus $B:=I+\mathbb C1$ is a $\Gamma$-invariant unital
$C^*$-subalgebra and $\delta_B$ is an ergodic IRA.
If $I\neq0$ then $B\neq\mathbb C1$: otherwise $0\neq I\subseteq\mathbb C1$
would imply $1\in I$ and hence $A=I=\mathbb C1$.

\medskip

\noindent$(1)$ Following the argument in
\cite[proof of Lemma~6.1]{AmrutamJiang2026}, let $I$ be a non-zero
proper closed two-sided ideal. If $B=A$, then $A/I\simeq\C$, so the quotient map defines a
character of $A$. Hence $\mathbb C1\subsetneq B\subsetneq A$, contradicting
the assumptions in $(1)$.

\medskip

\noindent$(2)$ Let $I$ be a non-zero nuclear closed two-sided ideal.
The quotient $B/I$ is either $\mathbb C$ or $0$, so $B$ is nuclear by
\cite[Proposition~IV.3.1.3]{Blackadar2006}.
Thus $\delta_B\neq\delta_{\mathbb C1}$ is a nuclear IRA, a contradiction.\end{proof}

\begin{remark}
In Proposition~\ref{prop: simple implies trivial IRAs}(1) the following holds.
\begin{itemize}
\item The absence of characters is necessary.
\end{itemize}
Indeed, for $A=\mathbb C^2$ with the trivial action, $\Sub(A)=\{\mathbb C1,A\}$, so the only ergodic IRAs are the two Dirac measures, although $A$ is not simple.
\begin{itemize}
\item The converse fails.
\end{itemize}
Indeed, the algebra $A=C_r^*(\mathbb F_2)$ is simple by \cite{Powers75} and has no characters. The non-trivial proper normal subgroup $[\mathbb F_2,\mathbb F_2]$ gives
$\mathbb C1\subsetneq C_r^*([\mathbb F_2,\mathbb F_2])\subsetneq A$. Hence, for the conjugation action of $\Gamma=\mathbb F_2$, the Dirac measure on this invariant subalgebra is a third ergodic IRA.
\end{remark}

\section{Exotic random subalgebras}\label{sec:exotic}

\noindent Throughout this section, $\Gamma$ is countable and
$\mu\in\Prob(\Gamma)$.

\subsection{Geometric and exotic random subalgebras}

\noindent We recall that the maps $\iota_r:H\mapsto C_r^*(H)$ and $\varphi:B\mapsto B''$ are Borel and equivariant by Theorem~\ref{thm:reduced-group-map} and
Corollary~\ref{cor:bi-commutant-map-borel}.

\begin{definition}\label{def:exotic-measures}
A random subalgebra $\nu\in\Prob(\Sub(C_r^*(\Gamma)))$ is called
\begin{itemize}
\item \textbf{geometric} if $\nu=(\iota_r)_*\eta$ for some $\eta\in\Prob(\Sub(\Gamma))$,
\item \textbf{exotic} if it is not geometric.
\end{itemize}
A random subalgebra $\nu\in\Prob(\Sub(\mathrm L(\Gamma)))$ is called
\textbf{super-exotic} if there is no
$\nu'\in\Prob(\Sub(C_r^*(\Gamma)))$ with $\varphi_*\nu'=\nu$.
\end{definition}

\noindent Let $X$ be a Polish space. A set $S\subseteq X$ is
\textbf{analytic} if $S=f(P)$ for some Polish space $P$ and continuous map $f:P\to X$. Borel images of analytic
sets in Polish spaces are analytic \cite[Proposition~14.4(ii)]{Kechris1995}.
For a Polish space $Y$, we call a map $f:X\to Y$
\textbf{$\sigma(\Sigma_1^1)$-measurable} if, for every Borel set $B\subseteq Y$,
$f^{-1}(B)$ belongs to the $\sigma$-algebra generated by the analytic subsets of $X$.
A subset $E\subseteq X$ is \textbf{universally measurable}
if it belongs to the completion of the Borel $\sigma$-algebra with respect
to every $\nu\in\Prob(X)$. A map $f:X\to Y$ is universally measurable
if $f^{-1}(B)$ is universally measurable for every Borel set $B\subseteq Y$.
Every analytic subset of $X$ is universally measurable
\cite[Theorem~21.10]{Kechris1995}. Consequently, every $\sigma(\Sigma_1^1)$-measurable map is universally measurable. In the sequel, values of Borel probabilities on universally measurable sets are taken
in their completions. Set
$$
I_\Gamma=\iota_r(\Sub(\Gamma))\quad\text{and}\quad
J_\Gamma=\varphi(\Sub(C_r^*(\Gamma))).
$$
Thus $I_\Gamma$ is closed by Theorem~\ref{thm:reduced-group-map}, and
$J_\Gamma$ is analytic.

\begin{lemma}\label{lem:canonical-cstar-lift}\label{RmkConcentration}
The following holds.
\begin{enumerate}
\item A RA $\nu$ of $C_r^*(\Gamma)$ is geometric if and only if
$\nu(I_\Gamma)=1$. Its unique representing random subgroup is
$\eta=(\iota_r^{-1})_*(\nu|_{I_\Gamma})$.
\item One has $J_\Gamma=\{N\in\Sub(\mathrm L(\Gamma)):(N\cap C_r^*(\Gamma))''=N\}$. The map $s:N\mapsto N\cap C_r^*(\Gamma)$ is $\sigma(\Sigma_1^1)$-measurable and
equivariant, and $\varphi\circ s=\mathrm{id}$ on $J_\Gamma$.
A RA $\nu$ of $\mathrm L(\Gamma)$ is not super-exotic if and only if
$\nu(J_\Gamma)=1$; in that case $\varphi_*(s_*\nu)=\nu$.
\end{enumerate}
Both maps $\iota_r^{-1}$ and $s$ are $\Gamma$-equivariant hence preserve invariance, quasi-invariance and $\mu$-stationarity.
\end{lemma}

\begin{proof}
$(1)$. By Theorem~\ref{thm:reduced-group-map}, $\iota_r$ is a Borel
isomorphism onto $I_\Gamma$, which gives both assertions.

\medskip
\noindent$(2)$. We first show that $J_\Gamma=\{N\in\Sub(\mathrm L(\Gamma)):(N\cap C_r^*(\Gamma))''=N\}$. If $N\in J_\Gamma$ then $N=B''$ for some $B\in\Sub(C_r^*(\Gamma))$. Hence,
$B\subseteq N\cap C_r^*(\Gamma)\subseteq N$ so $(N\cap C_r^*(\Gamma))''=N$. Conversely, if $(N\cap C_r^*(\Gamma))''=N$ then
$N=\varphi(s(N))\in J_\Gamma$. This also shows that $\varphi\circ s=\mathrm{id}$ on $J_\Gamma$. Let us now show that $s$ is $\sigma(\Sigma_1^1)$-measurable. Note that the relation $\mathcal R=\{(N,a)\in\Sub(\mathrm L(\Gamma))\times C_r^*(\Gamma):E_N(a)=a\}$ is closed as the inverse image of $\{0\}$ under the continuous map $(N,a)\mapsto\|E_N(a)-a\|_2$ from $\Sub(\mathrm L(\Gamma))\times C_r^*(\Gamma)$ to $[0,\infty)$. Given any norm-open $U\subseteq C_r^*(\Gamma)$, the set
$\{N:s(N)\cap U\neq\emptyset\}$ is the projection of
$\mathcal R\cap(\Sub(\mathrm L(\Gamma))\times U)$, hence is analytic.
By Remark~\ref{rem:lower-wijsman-borel}, the sets
$\mathcal I_U\cap\Sub(C_r^*(\Gamma))$ generate the Wijsman Borel
$\sigma$-algebra on $\Sub(C_r^*(\Gamma))$. Their inverse images under $s$
are analytic by the preceding argument. This shows that $s$ is $\sigma(\Sigma_1^1)$-measurable. Finally, every $\varphi_*\rho$ gives full measure to $J_\Gamma$; conversely,
$\nu(J_\Gamma)=1$ and $\varphi\circ s=\mathrm{id}$ on $J_\Gamma$ give
$\varphi_*(s_*\nu)=\nu$.

\medskip

\noindent $\Gamma$-equivariance of $\iota_r^{-1}$ and $s$ is clear.\end{proof}

\subsection{Spanning exotic random subalgebras}

\begin{theorem}\label{prop:spanning-exotic-qira}\label{thm:quasi-invariant-exotic}
The following holds.
\begin{enumerate}
\item An exotic RA of $C_r^*(\Gamma)$ exists if and only if
$|\Gamma|\geq3$. In that case, there is an atomic spanning exotic
qIRA concentrated on abelian subalgebras, hence nuclear.
\item A super-exotic RA of $\mathrm L(\Gamma)$ exists if and only if
$\Gamma$ is infinite. In that case, there is an atomic spanning super-exotic
qIRA concentrated on two-dimensional subalgebras, hence amenable.
\end{enumerate}
\end{theorem}

\begin{proof}
\emph{Step 1: necessity in $(1)$ and $(2)$.}
If $|\Gamma|\leq2$, every unital subalgebra of $C_r^*(\Gamma)$ is a subgroup
algebra, so $I_\Gamma=\Sub(C_r^*(\Gamma))$. If $\Gamma$ is finite, then
$C_r^*(\Gamma)=\mathrm L(\Gamma)$, so $J_\Gamma=\Sub(\mathrm L(\Gamma))$.
Lemma~\ref{lem:canonical-cstar-lift} proves both necessary conditions.

\medskip
\noindent\emph{Step 2: a generating family for $(1)$.}
Assume $|\Gamma|\geq3$. We show that there is  a sequence
$(B_n)_{n\geq1}$ of unital abelian subalgebras such that
$$
B_n\notin I_\Gamma\quad\text{for all}\quad n\geq1\quad\text{and}\quad C^*(B_n:n\geq1)=C^*_r(\Gamma).
$$
Indeed, let $A=C^*\Bigl(\bigcup\{C^*(1,x):x=x^*\in C^*_r(\Gamma),\ C^*(1,x)\notin I_\Gamma\}\Bigr)$. We first prove that $A=C^*_r(\Gamma)$. We use the following implication on Fourier coefficients.
$$
\text{For all }K\in\Sub(\Gamma)\text{ and }x\in C_r^*(K),\quad \widehat x(g):=\tau_0(\lambda_g^*x)\neq0
\Longrightarrow g\in K
$$
and we separate the proof in two cases: $\Gamma$ non-abelian and $\Gamma$ abelian.

\smallskip
\noindent\emph{Non-abelian $\Gamma$.}
Let $g\in\Gamma$. Note that the map $x\mapsto\widehat x(g)$ on $C^*_r(\Gamma)_{\mathrm{sa}}$ is continuous and real linear,
and is nonzero since $\widehat{\lambda_g+\lambda_g^*}(g)\geq1$.
Its kernel is therefore a proper closed real linear subspace, hence has
empty interior in $C^*_r(\Gamma)_{\mathrm{sa}}$. Hence, the complement of its kernel $\mathcal{H}_g:=\{x\in C^*_r(\Gamma):x=x^*,\ \widehat x(g)\neq0\}$ is open and dense in $C^*_r(\Gamma)_{\mathrm{sa}}$. Choose $a,b\in\Gamma$ with $ab\neq ba$. The set
$\mathcal{H}_a\cap\mathcal{H}_b$ is then norm dense in $C^*_r(\Gamma)_{\mathrm{sa}}$ by the previous discussion.  For $x\in\mathcal{H}_a\cap\mathcal{H}_b$, if there is $K\in\Sub(\Gamma)$ such that $C^*(1,x)=C_r^*(K)$ then
$a,b\in K$, contradicting commutativity of $C^*(1,x)$. It follows that $x\in A$. Hence $\mathcal{H}_a\cap\mathcal{H}_b\subseteq A$ and by density we we deduce that $C^*_r(\Gamma)_{\mathrm{sa}}\subseteq A$ so $C^*_r(\Gamma)=A$.

\smallskip
\noindent\emph{Abelian $\Gamma$.}
We prove $\lambda_g\in A$ for every $g\in\Gamma$.
For every finite subgroup $K\leq\Gamma$ with $|K|\geq3$, its minimal projections
$$
p_\chi=\frac1{|K|}\sum_{k\in K}\overline{\chi(k)}\lambda_k
\quad(\chi\in\widehat K)
$$
satisfy $\widehat{p}_\chi(k)\neq 0$ for all $k\in K$.
If $C^*(1,p_\chi)=C_r^*(H)$ for some $H\in\Sub(\Gamma)$ then $K\leq H$, whereas
$\dim C^*(1,p_\chi)=2<|K|$. Hence $C^*(1,p_\chi)\notin I_\Gamma$ so $p_\chi\in A$
for every $\chi$. Hence, $C_r^*(K)\subseteq A$ (in particular $\lambda_g\in A$ for every $g\in\Gamma$ with finite order $\geq 3$).

\smallskip
\noindent If $g$ has infinite order, both
$C^*(1,\operatorname{Re}\lambda_g)$ and
$C^*(1,\operatorname{Im}\lambda_g)$ are proper subalgebras of
$C_r^*(\langle g\rangle)\simeq C(\mathbb T)$, since neither real nor
imaginary part separates points of $\mathbb T$.
Their generators have nonzero $g$-Fourier coefficients, so neither algebra
belongs to $I_\Gamma$. Hence $\lambda_g=\operatorname{Re}\lambda_g+i\operatorname{Im}\lambda_g\in A$. This adapts
\cite[Example~3.5]{AmrutamJiang2023}.

\smallskip
\noindent It remains to consider an involution $g\neq e$ not contained in
a finite subgroup of order at least $3$.
Choose $h\notin\{e,g\}$. Then $h$ has infinite order, since otherwise
$\langle g,h\rangle$ would be finite of order at least $3$.
Put $r=\operatorname{Re}\lambda_h$ and let $\theta\in\operatorname{Aut}(C_r^*(\Gamma))$
be defined by $\theta(\lambda_k)=\lambda_{k^{-1}}$. $\theta$ is well defined since $\Gamma$ is abelian. Then $\theta(r+\lambda_g)=r+\lambda_g$ but $\theta(\lambda_h)\neq\lambda_h$,
so $\lambda_h\notin C^*(1,r+\lambda_g)$.
Since $\widehat{r+\lambda_g}(h)=1/2$, it follows that
$C^*(1,r+\lambda_g)\notin I_\Gamma$.
As $r\in A$, also $\lambda_g\in A$. This proves $A=C^*_r(\Gamma)$.

\smallskip
\noindent We can now finish the proof of Step $2$. Put $\mathcal F=\{B\in\Sub(A)\setminus I_\Gamma:B\text{ is abelian}\}$. Since $\Gamma$ is countable, there is a sequence $(a_j)_{j\geq1}\subseteq C^*_r(\Gamma)$ with $\overline{\{a_j:j\geq1\}}=C^*_r(\Gamma)$. By the first part of the proof, $C^*_r(\Gamma)=C^*(\bigcup_{B\in\mathcal F}B)$. Choose a finite subset $\mathcal F_{j,k}\subseteq\mathcal F$ such that $d\bigl(a_j,C^*(\bigcup_{B\in\mathcal F_{j,k}}B)\bigr)<1/k$ for all $j,k\geq1$. For $\{B_n:n\geq1\}:=\bigcup_{j,k\geq1}\mathcal F_{j,k}$ one has $a_j\in C^*(B_n:n\geq1)$ for all $j\geq 1$ so $C^*(B_n:n\geq1)=A$.

\medskip
\noindent\emph{Step 3: a generating family for $(2)$.}
Assume $\Gamma$ is infinite. We show that there is a sequence $(B_n)_n\subset\Sub({\rm L}(\Gamma))$ such that 
$$
\dim B_n=2,\quad B_n\notin J_\Gamma\text{ for all }n\geq1,\quad
W^*(B_n:n\geq1)={\rm L}(\Gamma).
$$
Since ${\rm L}(\Gamma)$ has separable predual, choose nonzero projections
$(p_n)_{n\geq1}$ with $W^*(p_n:n\geq1)={\rm L}(\Gamma)$. Since ${\rm L}(\Gamma)$ is diffuse, each corner $p_n{\rm L}(\Gamma)p_n$ also is and so contains an increasing family
$(q_t)_{0<t<\tau_0(p_n)}$ of projections with $\tau_0(q_t)=t$.
Since $\|q_s-q_t\|=1$ for $s\neq t$, norm separability of $C^*_r(\Gamma)$ implies
that at most countably many $q_t$ belong to $C^*_r(\Gamma)$.
For every $k\geq1$, we may therefore choose $q_{n,k}\leq p_n$ with
$$
q_{n,k}\notin C^*_r(\Gamma)\quad\text{and}\quad
\|p_n-q_{n,k}\|_2^2=\tau_0(p_n-q_{n,k})<1/k.
$$
Thus $p_n\in W^*(q_{n,k}:k\geq1)$ for every $n$, and
$W^*(q_{n,k}:n,k\geq1)={\rm L}(\Gamma)$.
Set $B_{n,k}=\mathbb Cq_{n,k}\oplus\mathbb C(1-q_{n,k})$. Let $x\in B_{n,k}\cap C^*_r(\Gamma)$ and write $x=aq_{n,k}+b(1-q_{n,k})\in C^*_r(\Gamma)$. If $x\notin\C1$ then $a\neq b$ and $q_{n,k}=\frac{aq_{n,k}+b(1-q_{n,k})-b1}{a-b}\in C^*_r(\Gamma)$, a contradiction. Hence
$B_{n,k}\cap C^*_r(\Gamma)=\mathbb C1$ so $(B_{n,k}\cap C^*_r(\Gamma))''=\mathbb C1\neq B_{n,k}$. Lemma~\ref{lem:canonical-cstar-lift}$(2)$ gives $B_{n,k}\notin J_\Gamma$.
By enumerating this family we obtain the desired sequence $(B_n)_{n\geq1}$.

\medskip
\noindent\emph{Step 4: the measures in $(1)$ and $(2)$.}
Use the family from Step~2 with $(D,E)=(C_r^*(\Gamma),I_\Gamma)$,
or the family from Step~3 with $(D,E)=(\mathrm L(\Gamma),J_\Gamma)$.
As in Remark~\ref{RmkNucRA}\eqref{item:spanning-nuclear-qira}, choose
$m\in\Prob(\Gamma)$ with $m(g)>0$ for every $g$, and define
$$
\nu=\sum_{n\geq1}\sum_{g\in\Gamma}2^{-n}m(g)
\delta_{\lambda_gB_n\lambda_g^*}\quad\text{and}\quad
\mathcal D=\{\lambda_gB_n\lambda_g^*:n\geq1,\ g\in\Gamma\}.
$$
The measure $\nu$ is atomic, and every point of the $\Gamma$-invariant countable set
$\mathcal D$ has positive mass. Thus, for every Borel set $F\subseteq\Sub(D)$
and $h\in\Gamma$,
$$
\nu(F)=0\iff F\cap\mathcal D=\emptyset
\iff h^{-1}F\cap\mathcal D=\emptyset\iff h\nu(F)=0.
$$
Hence $\nu$ is quasi-invariant. Since $E$ is invariant and $B_n\notin E$,
we have $\nu(E)=0$; Lemma~\ref{lem:canonical-cstar-lift} gives exoticity
in $(1)$ and super-exoticity in $(2)$.

\smallskip
\noindent If $C\in\Sub(D)$ satisfies $\nu(\Sub(C))=1$, then
$B_n\subseteq C$ for every $n$, since $\nu(\{B_n\})\geq2^{-n}m(e)>0$.
The generating property gives $C=D$, so Proposition~\ref{prop:alg sup unique}
implies $\langle\nu\rangle=D$.
In $(1)$, every atom is abelian, hence nuclear.
In $(2)$, every atom is two-dimensional, hence amenable.\end{proof}

\begin{remark}\label{PropSuperExoticDense}
Let $(D,E)\in\{(C_r^*(\Gamma),I_\Gamma),(\mathrm L(\Gamma),J_\Gamma)\}$
and let $\mathcal P$ be one of the classes RA, IRA, qIRA or $\mu$-SRA on $D$.
The sets
$$\mathcal X=\{\nu\in\mathcal P:\nu(E)<1\}\quad\text{and}\quad\mathcal X_{\mathrm{sp}}=\{\nu\in\mathcal X:\langle\nu\rangle=D\}$$
are either both empty or both dense in $\mathcal P$ for the weak topology.
Thus, if $\mathcal P$ contains an exotic measure for $D=C_r^*(\Gamma)$,
resp. a super-exotic measure for $D=\mathrm L(\Gamma)$, then its spanning
exotic, resp. spanning super-exotic, measures are dense in $\mathcal P$.

\medskip

\noindent Indeed, by Lemma~\ref{lem:canonical-cstar-lift}, $\mathcal X$ is the set of exotic,
respectively super-exotic, measures in $\mathcal P$.
If $\mathcal X=\emptyset$, then $\mathcal X_{\mathrm{sp}}=\emptyset$.
Otherwise, fix $\rho\in\mathcal X$. For $\nu\in\mathcal P$ and
$0<\varepsilon<1/2$, put
$$
\nu_\varepsilon=(1-2\varepsilon)\nu+\varepsilon\rho+\varepsilon\delta_D.
$$
Note that $\delta_D\in\mathcal P$ and the classes RA, IRA, $\mu$-SRA and qIRA are convex hence $\nu_\varepsilon\in\mathcal{P}$. Moreover, since 
$\nu_\varepsilon(E)\leq1-\varepsilon(1-\rho(E))<1$ one has  $\nu_\varepsilon\in\mathcal X$. If $C\in\Sub(D)$ and $\nu_\varepsilon(\Sub(C))=1$ then, by positivity, $\delta_D(\Sub(D)\setminus\Sub(C))=0$ which implies that $D\subseteq C$. Hence $C=D$, and Proposition~\ref{prop:alg sup unique} gives $\langle\nu_\varepsilon\rangle=D$ so $\nu_\varepsilon\in\mathcal X_{\mathrm{sp}}$. Since $\nu_\varepsilon\to\nu$ weakly this shows the Remark.\end{remark}

\subsection{Non-atomic exotic invariant random subalgebras}

The following tensor-product construction is a variant of
\cite[Example~4.3]{AHO25}.

\begin{lemma}\label{lem:bernoulli-exotic}
Let $H$ be countable and
$\Gamma=H^{(\mathbb Z)}\rtimes\langle t\rangle=H\wr\mathbb Z$.
Consider either
$$
(D,\mathcal D)=(C_r^*(H),C_r^*(\Gamma))
\quad\text{or}\quad(D,\mathcal D)=(\mathrm L(H),\mathrm L(\Gamma)).
$$
Let $\rho_k:D\to\mathcal D$ be the canonical embedding into the $k$th-coordinate
subgroup $H_k\leq H^{(\Z)}\leq\Gamma$, with $tH_kt^{-1}=H_{k+1}$.
Fix $B\in\Sub(D)\setminus\{\mathbb C1\}$ invariant under conjugation by $H$,
and put $B_0=\mathbb C1$, $B_1=B$.
For $\omega\in\Omega:=\{0,1\}^{\mathbb Z}$, define
$$
A_\omega=
\begin{cases}
C^*(\rho_k(B_{\omega_k}):k\in\mathbb Z),&D=C_r^*(H),\\
W^*(\rho_k(B_{\omega_k}):k\in\mathbb Z),&D=\mathrm L(H).
\end{cases}
$$
The map $\Theta:\omega\mapsto A_\omega$ is Borel and injective, and
$$
\nu=\Theta_*\beta\quad\text{where}\quad
\beta=(\tfrac12\delta_0+\tfrac12\delta_1)^{\otimes\mathbb Z},
$$
is a non-atomic ergodic IRA. If if $D=C_r^*(H)$ and $B\notin I_H$ then $\nu(I_\Gamma)=0$, hence $\nu$ is exotic. If $D=\mathrm L(H)$ and $B\notin J_H$ then $\nu(J_\Gamma)=0$, hence $\nu$ is super exotic.\end{lemma}

\begin{proof}
Let $G:=H^{(\mathbb Z)}$ and, for $\omega\in\Omega$, define the dense unital $*$-subalgebra of $A_\omega$ (in the norm topology in the $C^*$-case or the $\sigma$-weak topology in the von Neumann case)
$$
\mathcal A_\omega:=\operatorname{span}\left\{
\prod_{k\in F}\rho_k(b_k):
F\subseteq\mathbb Z\text{ finite},\ b_k\in B_{\omega_k}\ \text{ for all }k\in F
\right\}.
$$
\noindent For $k\in\mathbb Z$, let
$E_k:{\rm L}(\Gamma)\to{\rm L}(H_k)$ be the canonical trace-preserving
conditional expectation. Thus $E_k(\lambda_g)=1_{H_k}(g)\lambda_g$ for all $g\in\Gamma$ and $E_k(C_r^*(\Gamma))=C_r^*(H_k)$.
For $b=\prod_{l\in F}\rho_l(b_l)\in\mathcal A_\omega$, where
$F\subseteq\mathbb Z$ is finite and $b_l\in B_{\omega_l}$, put $b_l=1$
for $l\notin F$. Then
$$
E_k(b)=\left(\prod_{l\in F\setminus\{k\}}\tau_0(b_l)\right)\rho_k(b_k).
$$ Norm continuity in the
$C^*$-case, and normality in the von Neumann case, therefore give $E_k(A_\omega)\subseteq\rho_k(B_{\omega_k})$.
The reverse inclusion holds because $\rho_k(B_{\omega_k})\subseteq A_\omega$
and $E_k$ fixes this subalgebra. Since $B\neq\mathbb C1$, the equality $E_k(A_\omega)=\rho_k(B_{\omega_k})$ (for all $k$)
implies that $\Theta$ is injective.

\medskip

\noindent We now show measurability. In the $C^*$-case, we show that $\Theta$ is continuous for the lower topology, hence Borel
by Lemma~\ref{lem:lower-to-borel}. To do so, let $U\subseteq C_r^*(\Gamma)$ be a non-empty norm-open set and let us show that $\Theta^{-1}(\mathcal{I}_U)=\{\omega\in\Omega:A_\omega\cap U\neq \emptyset\}$ is open in $\Omega$. Let $\omega\in\Omega$ be such that $A_\omega\cap U\neq\emptyset$. Since $\mathcal{A}_\omega$ is norm dense in $A_\omega$, choose an element
$b=\sum_{j=1}^m\prod_{k\in F}\rho_k(b_{j,k})\in\mathcal A_\omega\cap U$,
where $F\subseteq\mathbb Z$ is finite and $b_{j,k}\in B_{\omega_k}$. Then, the set $\mathcal{V}:=\{\omega'\in\Omega:\omega'|_F=\omega|_F\}$ is open in $\Omega$ and $b\in \mathcal{A}_{\omega'}\cap U$ for all $\omega'\in\mathcal{V}$. Hence $\mathcal{V}\subseteq\Theta^{-1}(\mathcal{I}_U)$.

\medskip

\noindent In the von Neumann case, we use the argument from \cite[proof of Lemma~4.10]{AmrutamFimaJiang2026}. Identify $\mathrm L(G)=\overline{\bigotimes}_{k\in\mathbb Z}(\mathrm L(H),\tau_0)$ and write $E_\omega:{\rm L}(G)\to A_\omega$ the canonical trace preserving conditional expectation. For finite $F\subseteq\mathbb Z$ and $x_k\in\mathrm L(H)$ one has
$$
E_{\omega}\left(\prod_{k\in F}\rho_k(x_k)\right)
=\prod_{k\in F}\rho_k\bigl(E_{B_{\omega_k}}(x_k)\bigr).
$$
The right-hand side is locally constant in $\omega$. Since finite tensors are $2$-norm dense and all conditional expectations are $2$-norm contractions, the map
$\omega\mapsto E_{\omega}(x)$ is $2$-norm continuous for all $x\in\mathrm L(G)$. Hence the map $\Omega\to\Sub({\rm L}(G))$ is continuous. Since the inclusion $\Sub({\rm L}(G))\subset\Sub({\rm L}(\Gamma))$ is also continuous, the map $\Theta$ is continuous.

\medskip

\noindent Define $S:\Omega\to\Omega$ by $(S\omega)_k=\omega_{k-1}$. The $H$-invariance of $B$ gives $\lambda_gA_\omega\lambda_g^*=A_\omega$ for all $g\in G$ and $\lambda_tA_\omega\lambda_t^*=A_{S\omega}$. Thus $\Theta$ is equivariant for the $\Gamma$-action on $\Omega$
factoring through $\Gamma/G\simeq\mathbb Z$ and generated by $S$. Since $S$ is $\beta$-preserving and mixing (hence ergodic), the measure $\nu=\Theta_*\beta$ is invariant
and ergodic. Moreover, injectivity of $\Theta$
and non-atomicity of $\beta$ imply that $\nu$ is non-atomic.

\medskip

\noindent It remains to show exoticity (resp. super-exoticity). The set
$$\Omega_1=\{\omega:\omega_k=1\text{ for some }k\in\mathbb Z\}$$
has $\beta$-measure one. Fix $\omega\in\Omega_1$ and $k$ with $\omega_k=1$.
In the $C^*$-case, if $A_\omega=C_r^*(K)$ for a subgroup $K\leq\Gamma$, then $\rho_k(B)=E_k(A_\omega)=E_k(C_r^*(K))=C_r^*(K\cap H_k)$
contradicting $B\notin I_H$ after identifying $H_k$ with $H$. Hence, $\Theta(\Omega_1)\subseteq\Sub(C^*_r(\Gamma))\setminus I_\Gamma$ so $\nu(I_\Gamma)=0$. In the von Neumann case, suppose $A_\omega=C''$ for some
$C\in\Sub(C_r^*(\Gamma))$. Then, $E_k(C)\subseteq E_k(A_\omega)=\rho_k(B_{\omega_k})=\rho_k(B)$ and $E_k(C)\subseteq E_k(C^*_r(\Gamma))= C_r^*(H_k)=\rho_k(C^*_r(H))$ so $E_k(C)\subseteq \rho_k(B\cap C^*_r(H))$. Since $C$ is $\sigma$-weakly dense in $C''=A_\omega$ and both $E_k$ and $\rho_k$ are normal,
$$
\rho_k(B)=E_k(A_\omega)
\subseteq\overline{E_k(C)}^{\,\mathrm{\sigma-w}}
\subseteq\rho_k\!\left((B\cap C_r^*(H))''\right)
\subseteq\rho_k(B).
$$
Thus $(B\cap C_r^*(H))''=B$, contradicting $B\notin J_H$. Hence, $\Theta(\Omega_1)\subseteq\Sub({\rm L}(\Gamma))\setminus J_\Gamma$
so $\nu(J_\Gamma)=0$.\end{proof}

\begin{corollary}\label{cor:nonatomic-exotic}
The following holds.
\begin{enumerate}
\item If $H$ is countable and not $C^*$-simple, then
$C_r^*(H\wr\mathbb Z)$ has a non-atomic ergodic exotic IRA.
\item $\mathrm L(\mathbb Z\wr\mathbb Z)$ has a non-atomic ergodic
super-exotic IRA.
\end{enumerate}
\end{corollary}

\begin{proof}
$(1)$. Put $A=C_r^*(H)$ and choose a nonzero proper closed ideal $I$.
For the quotient map $q:A\to A/I$, the ideal $J=\ker(q\otimes\mathrm{id}_A)\triangleleft
A\otimes_{\min}A=C_r^*(H\times H)$
is nonzero, since $I\otimes1\subseteq J$, and proper, with
$(A\otimes_{\min}A)/J\simeq(A/I)\otimes_{\min}A
\supseteq1\otimes A\simeq A\neq\mathbb C$.
By \cite[Lemma~6.1]{AmrutamJiang2026}, $B=J+\mathbb C1$ is invariant and
not a subgroup algebra. Lemma~\ref{lem:bernoulli-exotic} gives a non-atomic
ergodic exotic IRA $\nu_0$ of $C_r^*((H\times H)\wr\mathbb Z)$.

\medskip
\noindent Let $\Gamma=H\wr\mathbb Z$ and let $t$ generate its $\mathbb Z$-factor and consider the normal index $2$ subgroup $\Lambda:=H^{(\Z)}\rtimes2\Z$ of $\Gamma$. The isomorphism
$(h_k)_{k\in\mathbb Z}\mapsto((h_{2n},h_{2n+1}))_{n\in\mathbb Z}$
from $H^{(\mathbb Z)}$ onto $(H\times H)^{(\mathbb Z)}$ intertwines
conjugation by $t^2$ with the shift by one coordinate, hence identifies
$\Lambda=H^{(\mathbb Z)}\rtimes2\mathbb Z
\simeq(H\times H)\wr\mathbb Z$.
Using this isomorphism, view the preceding measure $\nu_0$ on $\Sub(C^*_r(\Lambda))$ and then on $\Sub(C_r^*(\Gamma))$ via the Borel
inclusion of Proposition~\ref{rem:subspace-closed}, and set
$\nu=\tfrac12(\nu_0+t_*\nu_0)$. Since $\nu_0$ is $\Lambda$-invariant and
$t^{-1}\Lambda t=\Lambda$, for every $\lambda\in\Lambda$, $\lambda_*(t_*\nu_0)=t_*\bigl((t^{-1}\lambda t)_*\nu_0\bigr)=t_*\nu_0$, so $\lambda_*\nu=\nu$. Moreover, $t^2\in\Lambda$ gives
$t_*\nu=\tfrac12(t_*\nu_0+(t^2)_*\nu_0)=\nu$.
As $\Gamma=\langle\Lambda,t\rangle$, $\nu$ is $\Gamma$-invariant. $\nu$ is non-atomic since $\nu_0$ is. If $E$ is a $\Gamma$-invariant Borel set then it is $\Lambda$-invariant and since $\nu_0$ is ergodic for the $\Lambda$-action we get $\nu(E)=\tfrac12(\nu_0(E)+\nu_0(t^{-1}\!\cdot E))=\nu_0(E)\in\{0,1\}$ so $\nu$ is ergodic. Since $I_\Gamma$ is $\Gamma$-invariant, $\nu_0$ is supported on $\Sub(C^*_r(\Lambda))$ and $I_\Gamma\cap \Sub(C^*_r(\Lambda))=I_\Lambda$ one has $\nu(I_\Gamma)=\nu_0(I_\Gamma)=\nu_0(I_\Lambda)=0$. This shows the exoticity of $\nu$.

\medskip
\noindent$(2)$. Let $m$ be the Haar probability measure on $\mathbb T$. Under the Fourier
transform, view  $C_r^*(\mathbb Z)=C(\mathbb T)$ and $\mathrm L(\mathbb Z)=\mathrm L^\infty(\mathbb T,m)$. Fix any $m$-measurable $E\subseteq\mathbb T$ with $0<m(E)<1$, and define
$B:=\mathbb C1_E\oplus\mathbb C1_{E^c}.
$
If $h\in B\cap C(\mathbb T)$, there are $a,b\in\mathbb C$ such that
$h=a1_E+b1_{E^c}$ almost everywhere. Thus $(h-a)(h-b)=0$ almost
everywhere, hence everywhere by continuity and since $m$ has full support.
Since $\mathbb T$ is connected, $h$ must be constant. Hence,
$B\cap C(\mathbb T)=\mathbb C1$ and $(B\cap C(\mathbb T))''\neq B$, so $B\notin J_{\mathbb Z}$ by Lemma~\ref{lem:canonical-cstar-lift}$(2)$.
Conjugation by $\mathbb Z$ is trivial, so Lemma~\ref{lem:bernoulli-exotic} applies.
\end{proof}

\begin{remark}The measures in Lemma~\ref{lem:bernoulli-exotic} and
Corollary~\ref{cor:nonatomic-exotic} are concentrated on subalgebras of $C^*_r(H^{(\mathbb Z)})$ (or ${\rm L}(H^{(\mathbb Z)})$ in the von Neumann setting), hence they are not spanning in the ambient wreath-product algebra.\end{remark}

\section{Groups C*-algebras without exotic invariant random subalgebra}\label{sec:noexotic}

\noindent In this section we produce examples of groups $\Gamma$ for which $C^*_r(\Gamma)$ has no exotic IRA. Our main result is Corollary \ref{cor:hyperbolic_no_exotic} in which we prove that the reduced C*-algebra of any torsion-free non-amenable hyperbolic group has no exotic IRA.

\vspace{0.2cm}

\noindent The following definition is \cite[Definition~2.1]{GL25}, but we do not assume the space to be compact.

\begin{definition}
Let $\alpha:\Gamma\curvearrowright X$ be an action by homeomorphisms
on a topological space $X$, and write $sx=\alpha_s(x)$. An element $s\in\Gamma$ is a \textbf{recurrent direction} for $x\in X$ if
for every open neighbourhood $U$ of $x$ and every $N\geq1$, there exists
$n\geq N$ such that $s^n x\in U$. Write
$\mathrm R_\Gamma(x)=\{s\in\Gamma:s\text{ is a recurrent direction for }x\}$. A point $x$ is a \textbf{boomerang point} if $\mathrm R_\Gamma(x)=\Gamma$.
\end{definition}

\begin{remark}
Suppose that $x$ has a countable neighbourhood basis. Then
$s\in\mathrm R_\Gamma(x)$ if and only if there are positive integers
$m_j\uparrow\infty$ such that $s^{m_j}x\to x$.
Indeed, choose a countable open neighbourhood basis $(U_j)$ of $x$ and set
$V_j=\bigcap_{i=1}^jU_i$. If $s\in\mathrm R_\Gamma(x)$, starting with
$m_0=0$, choose successively $m_j>m_{j-1}$ with $s^{m_j}x\in V_j$. By construction $s^{m_j}x\to x$. The converse implication is clear and does not require any countability assumption.
Such a sequence $(m_j)$ is called a \textbf{return sequence for $s$ at $x$}.
\end{remark}

\noindent In the sequel we write $\mathrm{Boom}_\alpha(X)=\{x\in X:\mathrm R_\Gamma(x)=\Gamma\}$.

\begin{proposition}\label{prop:ira-boomerang}
Let $\Gamma$ be a countable group and let $\alpha:\Gamma\curvearrowright X$
be an action by homeomorphisms on a metrizable space $X$. Then:
\begin{enumerate}
\item $\mathrm{Boom}_\alpha(X)$ is a $G_\delta$ subset of $X$.
\item If $X$ is moreover separable, every $\Gamma$-invariant Borel probability measure $\nu$ on $X$
satisfies $\nu(\mathrm{Boom}_\alpha(X))=1$.
\end{enumerate}
\end{proposition}

\begin{proof}
Fix a compatible metric $d$ on $X$.

\vspace{0.2cm}

\noindent$(1)$. It suffices to write
$\mathrm{Boom}_\alpha(X)
=\bigcap_{s\in\Gamma}\bigcap_{k,N\geq1}\ \bigcup_{n\geq N}
\{x\in X:d(s^n x,x)<2^{-k}\}$.

\vspace{0.2cm}

\noindent$(2)$. We use the Poincar\'e recurrence argument of
\cite[Example~2.9(5)--(6) and its proof]{GL25}. Since $X$ is separable and metrizable, choose a countable base
$(U_\ell)_{\ell\geq1}$ of $X$. For $s\in\Gamma$ and $\ell\geq1$, the set
$N_{s,\ell}:=U_\ell\setminus
\bigcap_{N\geq1}\bigcup_{n\geq N}\alpha_s^{-n}(U_\ell)$ is $\nu$-null by Poincar\'e recurrence. Hence
$N_0=\bigcup_{s\in\Gamma}\bigcup_{\ell\geq1}N_{s,\ell}$ is $\nu$-null.
If $x\notin N_0$, fix $s\in\Gamma$ and, for each $j\geq1$, choose
$U_{\ell_j}$ with $x\in U_{\ell_j}\subseteq\{y\in X:d(y,x)<2^{-j}\}$. There are infinitely many $n$ such that $s^n x\in U_{\ell_j}$.
Starting with $m_0=0$, choose successively $m_j>m_{j-1}$ among these
return times. Then $d(s^{m_j}x,x)<2^{-j}$, so $s^{m_j}x\to x$.
Thus $X\setminus N_0\subseteq\mathrm{Boom}_\alpha(X)$, proving (2).
\end{proof}

\noindent The following is the $C^*$-algebra analogue of \cite[Definitions~2.6]{GL25}.

\begin{definition}
Let $\alpha:\Gamma\curvearrowright A$ be an action of a group $\Gamma$ on a unital C*-algebra $A$. A subalgebra $B\in\Sub(A)$ is called
\textbf{boomerang} if it is a boomerang point of the associated action $\Gamma\curvearrowright\Sub(A)$, where $\Sub(A)$ has the Wijsman topology.
\end{definition}

\noindent We now focus on the group case $A=C^*_r(\Gamma)$, for a countable group $\Gamma$, with $\alpha$ the action by conjugation by elements $\lambda_g$. The following Lemma is a direct consequence of the results of \cite{AmrutamJiang2026}. An element $g\in\Gamma\setminus\{e\}$ is called \textbf{primitive} if for all $h\in\Gamma$ and $n\in\Z$, $g=h^n$ implies $n\in\{-1,1\}$. In the remainder of this section, for a subgroup $H\leq\Gamma$, we denote by $E_H:C^*_r(\Gamma)\to C^*_r(H)$ the canonical conditional expectation.

\begin{lemma}\label{lem:expectation}
Let $\Gamma$ be a torsion-free non-amenable hyperbolic group. If $B\in\Sub(C^*_r(\Gamma))$ is boomerang and $t\neq e$ is a primitive element, then
$E_{\langle t\rangle}(B)\subseteq B$.
In particular, for all $w\in\Gamma\setminus\{ e\}$, the centralizer
$C_\Gamma(w)=\{g\in\Gamma:gw=wg\}$ satisfies $E_{C_\Gamma(w)}(B)\subseteq B$.
\end{lemma}

\begin{proof}
Fix $a\in B$ and a return sequence $m_j$ for $t$, so that
$\lambda_{t^{m_j}}B\lambda_{t^{-m_j}}\to B$ in the Wijsman topology. Hence $d(\lambda_{t^{-m_j}}a\lambda_{t^{m_j}},B)
=d(a,\lambda_{t^{m_j}}B\lambda_{t^{-m_j}})
\to d(a,B)=0$. Since $t^{-1}$ is primitive, \cite[Lemma~4.1]{AmrutamJiang2026} verifies the boundary
separation hypothesis of \cite[Proposition~3.2]{AmrutamJiang2026} for $t^{-1}$.
That proposition applies to every subsequence, hence to each $(m_j)_{j\geq J}$: for every $\varepsilon>0$ there are
$j_1,\ldots,j_q\geq J$ s.t.
$$
\left\|E_{\langle t\rangle}(a)
-\frac1q\sum_{i=1}^q\lambda_{t^{-m_{j_i}}}a\lambda_{t^{m_{j_i}}}\right\|
<\varepsilon.
$$
Choose $J$ so large that
$d(\lambda_{t^{-m_j}}a\lambda_{t^{m_j}},B)<\varepsilon$ for every $j\geq J$,
and choose $b_i\in B$ with
$\|\lambda_{t^{-m_{j_i}}}a\lambda_{t^{m_{j_i}}}-b_i\|<\varepsilon$.
Since $q^{-1}\sum_{i=1}^q b_i\in B$, the triangle inequality gives
$$
d(E_{\langle t\rangle}(a),B)
\leq\left\|E_{\langle t\rangle}(a)-\frac1q\sum_{i=1}^q b_i\right\|
<2\varepsilon.
$$
Letting $\varepsilon\to0$ and using closedness of $B$ yields
$E_{\langle t\rangle}(a)\in B$.

\vspace{0.2cm}

\noindent Fix $w\neq e$. Since $\Gamma$ is torsion-free, $w$ has infinite order and \cite[Corollary~III.$\Gamma$.3.10]{BH99}
gives $[C_\Gamma(w):\langle w\rangle]<\infty$.
Being torsion-free and virtually cyclic, $C_\Gamma(w)=\langle t\rangle\cong\Z$.
If $t=u^k$ then $u\in C_\Gamma(w)=\langle t\rangle$ so $u=t^\ell$
and $k\ell=1$. Thus $t$ is primitive and the first assertion gives
$E_{C_\Gamma(w)}(B)\subseteq B$.
\end{proof}

\begin{lemma}\label{lem:isolation}
Let $\Gamma$ be a countable group. Suppose that $t,q\in\Gamma$ freely generate a free group of rank two. If $r\in\Z\setminus\{0\}$ then, for all  $x\in C_r^*(\langle t\rangle)$ with
$\tau_0(x)=0$ one has $E_{C_\Gamma(t^rqt^rq^{-1})}(x\lambda_qx\lambda_q^*)=\tau_0(x\lambda_{t^r}^*)^2\lambda_{t^rqt^rq^{-1}}$.
\end{lemma}

\begin{proof}
Let $x\in\C[\langle t\rangle]$ with $\tau_0(x)=0$ and write $x=\sum_{k\in F}x_k\lambda_{t^k}$ for some finite subset $F\subset\Z\setminus\{0\}$ and where $x_k:=\tau_0(x\lambda_{t^k}^*)$. Let $k,l\neq0$ and put $w:=t^rqt^rq^{-1}$. By \cite[Lemma~2.3]{AmrutamJiang2026} we have
$t^kqt^lq^{-1}\in C_\Gamma(w)\iff k=l=r$ hence:
\begin{eqnarray*}
E_{C_\Gamma(w)}(x\lambda_qx\lambda_q^*)&=&\sum_{k,l\in F}x_kx_l E_{C_\Gamma(w)}(\lambda_{t^kqt^lq^{-1}})\\
&=&\sum_{k,l\in F, t^kqt^lq^{-1}\in C_\Gamma(w)}x_kx_l \lambda_{t^kqt^lq^{-1}}=   \tau_0(x\lambda_{t^r}^*)^2\lambda_w.
\end{eqnarray*}
\noindent By density of such $x$ in
$\ker(\tau_0|_{C_r^*(\langle t\rangle)})$ and by continuity we conclude the proof.\end{proof}

\noindent The following lemma is standard. We include a short proof for the sake of completeness.

\begin{lemma}\label{LemAsymptotic}
Let $A$ be a C*-algebra. If $B\subseteq A$ a sub-$C^*$-algebra then
$$\mathcal B:=\{(a_j)\in\ell^\infty(\N,A):d(a_j,B)\to0\}$$
is a C*-subalgebra of $\ell^\infty(\N,A)$. Moreover, the following holds.
\begin{enumerate}
\item For every sequence $(T_j)$ of linear contractions $T_j:A\to A$ there exists a unique linear map $T:\ell^\infty(\N,A)\to\ell^\infty(\N,A)$ such that $T(a_j)_j:=(T_j(a_j))_j$ for all $(a_j)_j\in\ell^\infty(\N,A)$. $T$ is a contraction and, if $T_j(B)\subseteq B$ for all $j$ then $T(\mathcal{B})\subseteq \mathcal{B}$.
\item For constant sequences, $(a,a,\ldots)\in\mathcal B\iff a\in B$.
\end{enumerate}
\end{lemma}

\begin{proof}
It is clear that $\mathcal B$ is a self-adjoint linear subspace. To prove norm closedness, let
$a^{(n)}=(a_j^{(n)})_j\in\mathcal B$ and
$a=(a_j)_j\in\ell^\infty(\N,A)$ satisfying
$\|a^{(n)}-a\|_\infty\longrightarrow0$. For every $n,j\in\N$,
$d(a_j,B)\leq\|a_j-a_j^{(n)}\|+d(a_j^{(n)},B)
\leq\|a-a^{(n)}\|_\infty+d(a_j^{(n)},B)$. Since $a^{(n)}\in\mathcal B$, taking the upper limit as $j\to\infty$ gives, for all $n\in\N$,
$$
0\leq\limsup_{j\to\infty}d(a_j,B)
\leq\|a-a^{(n)}\|_\infty.
$$
Letting $n\to\infty$ yields $d(a_j,B)\to0$, hence
$a\in\mathcal B$. Stability under multiplication follows from $d(uv,B)\leq\|u\|d(v,B)+\|v\|d(u,B)+d(u,B)d(v,B)$. Indeed, for $\varepsilon>0$, choose $b,c\in B$ such that
$\|u-b\|\leq d(u,B)+\varepsilon$ and
$\|v-c\|\leq d(v,B)+\varepsilon$. Since $bc\in B$,
$uv-bc=(u-b)v+b(v-c)$ and $\|b\|\leq\|u\|+d(u,B)+\varepsilon$, we have
\begin{eqnarray*}
d(uv,B)&\leq&\|u-b\|\|v\|+\|b\|\|v-c\|\\
&\leq&(d(u,B)+\varepsilon)\|v\|
 +(\|u\|+d(u,B)+\varepsilon)(d(v,B)+\varepsilon).
\end{eqnarray*}
Letting $\varepsilon\to0$ proves the inequality. Statements $(1)$ and $(2)$ are clear.\end{proof}

\begin{theorem}\label{thm:boomerang}
Let $\Gamma$ be a torsion-free non-amenable hyperbolic group and $B\in\Sub(C_r^*(\Gamma))$. If $B$ is boomerang then there exists $H\leq\Gamma$ such that $B=C_r^*(H)$.\end{theorem}

\begin{proof}
Put $A=C_r^*(\Gamma)$ and let $B\in\Sub(A)$ be a boomerang subalgebra. Define $H_B:=\{g\in\Gamma:\lambda_g\in B\}$. It is clear that $H_B$ is a subgroup of $\Gamma$ and $C^*_r(H_B)\subset B$. It suffices to show that $B\subset C^*_r(H_B)$. To do so, we will show the following statement: for all $a\in B$ and all $g\in\Gamma\setminus\{e\}$ one has $\tau_0(a\lambda_g^*)\neq0\Rightarrow g\in H_B$.

\vspace{0.2cm}

\noindent Fix $a\in B$ and $g\neq e$ such that $\tau_0(a\lambda_g^*)\neq0$. Replacing $a$ by $a-\tau_0(a)1$ we may and will assume that $\tau_0(a)=0$.
Write $g=t^r$, with $t$ primitive and $r\neq0$
\cite[Remark~4.2]{AmrutamJiang2026}. By Lemma~\ref{lem:expectation}, $x:=E_{\langle t\rangle}(a)\in B$. Note also that $\tau_0(x)=0$ and $\tau_0(x\lambda_g^*)\neq0$.

\vspace{0.2cm}

\noindent Choose $h\in\Gamma$ such that $t,h$ freely generate $\F_2$
\cite[Theorem~0.2]{AD19}. The boomerang property gives positive
integers $m_j\uparrow\infty$ with
$\lambda_{h^{m_j}}B\lambda_{h^{-m_j}}\to B$ in the Wijsman topology.
Since
$d(\lambda_{h^{-m_j}}x\lambda_{h^{m_j}},B)
=d(x,\lambda_{h^{m_j}}B\lambda_{h^{-m_j}})\to d(x,B)=0$,
we have $(\lambda_{h^{-m_j}}x\lambda_{h^{m_j}})_j\in\mathcal B$, where $\mathcal B:=\{(a_j)\in\ell^\infty(\N,A):d(a_j,B)\to0\}$, as in Lemma \ref{LemAsymptotic}. Also $(x)_j\in\mathcal B$, hence $(x\lambda_{h^{-m_j}}x\lambda_{h^{m_j}})_j\in\mathcal B$ since $\mathcal{B}$ is a subalgebra.

\vspace{0.2cm}

\noindent For every $j$, the pair $t,h^{-m_j}$ freely generates $\F_2$ so Lemma~\ref{lem:isolation} gives
$$
E_{C_\Gamma(gh^{-m_j}gh^{m_j})}
 (x\lambda_{h^{-m_j}}x\lambda_{h^{m_j}})
=\tau_0(x\lambda_g^*)^2\lambda_{gh^{-m_j}gh^{m_j}}.
$$
These conditional expectations preserve $B$ by
Lemma~\ref{lem:expectation} (since $gh^{-m_j}gh^{m_j}\neq e$) and Lemma \ref{LemAsymptotic}(1) implies that $(\tau_0(x\lambda_g^*)^2\lambda_{gh^{-m_j}gh^{m_j}})_j\in\mathcal B$. Dividing by the non-zero scalar $\tau_0(x\lambda_g^*)^2$ gives $
(\lambda_{gh^{-m_j}gh^{m_j}})_j\in\mathcal B$. Since $\mathcal B$ is a $*$-algebra, it follows that $(\lambda_{gh^{-m_j}gh^{m_j}}
 \lambda_{h^{-m_j}}x^*\lambda_{h^{m_j}})_j\in\mathcal B$.
 
\vspace{0.2cm}

\noindent Freeness gives $\langle t\rangle\cap
h^{-m_j}\langle t\rangle h^{m_j}=\{e\}$, and hence $E_{\langle t\rangle}(\lambda_{h^{-m_j}}y\lambda_{h^{m_j}})
=\tau_0(y)1$ for all  $y\in C_r^*(\langle t\rangle)$ (first for group unitaries and then by linearity, continuity and norm density). Therefore
$$
E_{\langle t\rangle}\!\left(
 \lambda_{gh^{-m_j}gh^{m_j}}
 \lambda_{h^{-m_j}}x^*\lambda_{h^{m_j}}\right)=\lambda_g E_{\langle t\rangle}
 (\lambda_{h^{-m_j}}\lambda_gx^*\lambda_{h^{m_j}})
 =\overline{\tau_0(x\lambda_g^*)}\lambda_g.
$$
By Lemma~\ref{lem:expectation} $E_{\langle t\rangle}(B)\subseteq B$. Hence, Lemma \ref{LemAsymptotic}(1) and the previous computation shows that the constant sequence
$(\overline{\tau_0(x\lambda_g^*)}\lambda_g)_j$ belongs to
$\mathcal B$, so $\lambda_g\in B$ and $g\in H_B$.\end{proof}

\begin{corollary}\label{cor:hyperbolic_no_exotic}
If $\Gamma$ is a torsion-free non-amenable hyperbolic group, then every IRA of $C_r^*(\Gamma)$ is geometric.
\end{corollary}

\begin{proof}
Let $\nu$ be an IRA of $C_r^*(\Gamma)$. Since $C_r^*(\Gamma)$ is separable,
Proposition~\ref{prop:ira-boomerang} implies that $\nu$-almost every $B$ is boomerang.
Theorem \ref{thm:boomerang} therefore gives $\nu(I_\Gamma)=1$ so $\nu$ is geometric, by Lemma \ref{lem:canonical-cstar-lift}.\end{proof}


\end{document}